\documentclass[11pt]{article}

\usepackage{amssymb}
\usepackage{amsmath}
\usepackage{amsthm}
\usepackage{epsfig}
\usepackage{amscd}
\usepackage[all]{xy}
\usepackage{enumerate}

\usepackage{color}

\newtheorem{lemma}{Lemma}[section]
\newtheorem{teo}[lemma]{Theorem}
\newtheorem{prop}[lemma]{Proposition}
\newtheorem{cor}[lemma]{Corollary}

\theoremstyle{definition}

\newtheorem{quest}[lemma]{Question}

\theoremstyle{remark}
\newtheorem{rem}[lemma]{Remark}

\newcommand{\matN}{\ensuremath {\mathbb{N}}}

\newcommand{\matQ} {\ensuremath {\mathbb{Q}}}
\newcommand{\matZ} {\ensuremath {\mathbb{Z}}}
\newcommand{\matC} {\ensuremath {\mathbb{C}}}

\newcommand{\matH} {\ensuremath {\mathbb{H}}}

\newcommand{\GL}{{\rm GL}}

\newcommand{\matr} [4] {\left(\begin{array}{@{}c@{\ }c@{}} #1 & #2 \\ #3 & #4 \\ \end{array} \right)}

\newcommand{\lens}[2]{L({\scriptstyle #1},{\scriptstyle #2})}

\newcommand{\seifuno}[3]{\big(#1,({\scriptstyle #2},{\scriptstyle #3})\big)}

\newcommand{\seifdue}[5]{\big(#1,({\scriptstyle #2},{\scriptstyle #3}),
                       ({\scriptstyle #4},{\scriptstyle #5})\big)}

\newcommand{\seiftre}[7]{\big(#1,({\scriptstyle #2},{\scriptstyle #3}),
                       ({\scriptstyle #4},{\scriptstyle #5}),
                       ({\scriptstyle #6},{\scriptstyle #7})\big)}

\newcommand{\bigu}[4]{\bigcup\nolimits_{{\tiny{\matr {#1} {#2} {#3} {#4}}}\phantom{\Big|}\!\!}}
\newcommand{\bigb}[4]{\big/_{{\tiny{\matr {#1} {#2} {#3} {#4}}}\phantom{\Big|}\!\!}}

\newcommand{\alza}{\phantom{\Big(}\!\!\!\!\!}

\newfont{\Got}{eufm10 scaled 1200}
\newcommand{\permu}{{\hbox{\Got S}}}

\author{Bruno \textsc{Martelli}\and Carlo \textsc{Petronio}\and Fionntan \textsc{Roukema}}

\title{Exceptional Dehn surgery\\
on the minimally twisted five-chain link}

\begin{document}

\maketitle

\begin{abstract}
\noindent
We consider in this paper the minimally twisted chain link with 5 components in the $3$-sphere,
and we analyze the Dehn surgeries on it, namely the Dehn fillings on its exterior $M_5$.
The $3$-manifold $M_5$ is a nicely symmetric hyperbolic one, filling which one gets a wealth of
hyperbolic $3$-manifolds having $4$ or fewer (including $0$) cusps.
In view of Thurston's
hyperbolic Dehn filling theorem it is then natural to face the problem of classifying
all the \emph{exceptional} fillings on $M_5$, namely those yielding non-hyperbolic $3$-manifolds.
Here we completely solve this problem, also showing that, thanks to the symmetries of $M_5$ and of some
hyperbolic manifolds resulting from fillings of $M_5$,
the set of exceptional fillings on $M_5$ is described by a very small
amount of information.\\
\textbf{MSC (2010):} 57M50 (primary), 57M25 (secondary).
\end{abstract}

\maketitle

\section*{Introduction} \label{introduction:section}

In this paper we establish the following main result. The terminology,
the context, the relevance, and a conceptual outline of the methods underlying
the proof are explained in the next few pages:

\begin{teo}\label{main:intro:teo}
Consider the link shown in Fig.~\ref{chain5_fig:fig}, fix a cyclic ordering on its
components and employ the meridian-longitude homology bases to parameterize the
Dehn surgeries on the link. Then a Dehn surgery is exceptional if and only if
up to a composition of the following maps
\begin{align*}
(\alpha_1, \alpha_2,\alpha_3,\alpha_4,\alpha_5) & \longmapsto (\alpha_5, \alpha_1, \alpha_2, \alpha_3, \alpha_4)\\
(\alpha_1, \alpha_2, \alpha_3, \alpha_4, \alpha_5) & \longmapsto (\alpha_5, \alpha_4, \alpha_3, \alpha_2, \alpha_1)\\
(\alpha_1, \alpha_2,\alpha_3,\alpha_4,\alpha_5) & \longmapsto
    \big(\tfrac1{\alpha_2}, \tfrac1{\alpha_1},1-\alpha_3,\tfrac{\alpha_4}{\alpha_4-1},1-\alpha_5\big)\\
\left(-1, \alpha_2,\alpha_3,\alpha_4,\alpha_5\right) & \longmapsto
    \left(-1, \alpha_3-1,\alpha_4,\alpha_5+1,\alpha_2\right)\\
\left(-1, -2, -2, -2, \alpha \right) & \longmapsto \left(-1, -2, -2, -2, -\alpha - 6 \right)
\end{align*}
it contains one of the next two or it is one of the subsequent five:
$$\begin{array}{rcr}
\infty \quad (-1, -2, -2, -1)  & \left(-2, -\tfrac 12, 3, 3, -\tfrac 12\right) & (-1, -2, -2, -3, -5)\\
(-1, -2, -3, -2, -4) & (-1, -3, -2, -2, -3) & (-2, -2, -2, -2, -2).
\end{array}$$
Moreover, no two of these seven fillings are related to each other by any composition of the above five maps.
\end{teo}
\begin{figure}
\begin{center}
\includegraphics[width = 4cm]{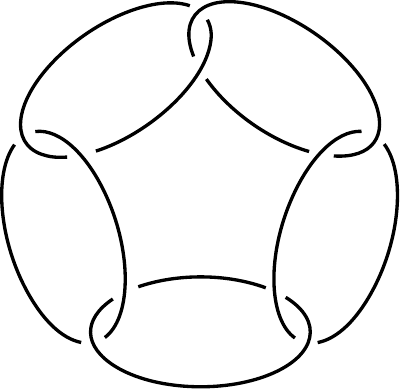}
 \end{center}
 \caption{The minimally twisted chain link with 5 components L10n113.}
 \label{chain5_fig:fig}
\end{figure}

\noindent
\textbf{Dehn surgery and Dehn filling}\quad
The operation of \emph{Dehn surgery} on a link in the $3$-sphere, and its natural generalization,
termed \emph{Dehn filling}, are fundamental ones in $3$-dimensional geometric topology. The input of an
operation of Dehn filling is given by
\begin{itemize}
\item a compact 3-manifold $M$ (that we will always tacitly assume to be
orientable and connected) with boundary $\partial M$ consisting of tori, and
\item a \emph{slope} (namely, the isotopy class of a nontrivial simple closed unoriented curve) on each component of $\partial M$.
\end{itemize}
The result of the operation is the manifold obtained by attaching to $M$ a copy of the solid torus
$D^2\times S^1$ along each component of $\partial M$,
with a solid torus attached to a component $T$ of $\partial M$ so
that the meridian $(\partial D^2)\times\{*\}$ is matched to the slope $\alpha$ contained in $T$.
To allow $T$ to be left unfilled, one also allows $\alpha$ to be the empty slope. A Dehn surgery
on a link $L$ in the $3$-sphere is a Dehn filling on
its exterior (the complement of an open regular neighbourhood).
To highlight the importance of this
operation we recall for instance the celebrated theorem of
Lickorish~\cite{Rolfsen}, according to which every closed (orientable) $3$-manifold is the
result of a Dehn surgery on some link in $S^3$.
\emph{In the rest of this paper, when a manifold $M$ bounded by tori is given, by a \emph{(Dehn) filling}
on $M$ we will mostly refer to the set of slopes along which the filling has to be performed.
Occasionally we will also use the same term to refer to the manifold resulting from the operation of
filling along the given slopes, but when there is any risk of confusion we will distinguish between the filling (viewed as an instruction) and its result.}

\medskip

\noindent
\textbf{Hyperbolic manifolds and exceptional fillings}\quad
We say that $M$ is \emph{hyperbolic}~\cite{BePe} if its interior admits a complete hyperbolic metric with finite volume,
in which case the components of $\partial M$ correspond to the \emph{cusps} of the interior of $M$. The
famous \emph{hyperbolic Dehn filling theorem}
of Thurston~\cite{Thu} states that if $M$ is hyperbolic then ``most'' Dehn fillings on $M$ give manifolds that are
also hyperbolic. A Dehn filling not producing a hyperbolic manifold is called \emph{exceptional}.
We continue here the program initiated in~\cite{MaPe} of classifying the exceptional fillings on hyperbolic
manifolds with an increasing number of boundary components. Namely, we provide a complete classification of all the
exceptional fillings on the exterior $M_5$ of the
\emph{minimally twisted chain link} with 5 components in the $3$-sphere, shown in Fig.~\ref{chain5_fig:fig}
and denoted by L10n113 in Thistlewaite's tables.

The relevance of the 5-cusped hyperbolic manifold $M_5$ comes from the following facts. First, it is conjecturally~\cite{Ago}
the 5-cusped manifold with smallest volume $10.149\ldots$ and smallest Matveev complexity $10$
(it can be triangulated using 10 regular ideal hyperbolic tetrahedra, see below). Second, by filling $M_5$ one obtains
a multitude of interesting manifolds, including most manifolds from the cusped census~\cite{CaHiWe}, and
many one-cusped manifolds having interesting exceptional fillings, such as various families of Berge~\cite{Bak}
knot exteriors, all Eudave-Mu\~noz~\cite{EM, GoLu} knot exteriors, and all exceptional reducible-toroidal pairs
at maximal distance~\cite{Kang}. Third, the manifold $M_5$ has a nice and unusually large
symmetry group, because it double covers
a very natural and symmetric orbifold, called the \emph{pentangle},
with underlying space $S^3$ viewed as the boundary of the 4-simplex
and as singular set the 1-skeleton of the 4-simplex.

The exceptional fillings on $M_5$ are classified by Theorem~\ref{main:intro:teo}. Our proof of this result was
computer-assisted but rigorous, namely immune from drawbacks coming from numerical approximation.
We used a python code named {\tt find\_exceptional\_fillings.py}, written by the first author and publicly available from \cite{M}. The code uses the \emph{SnapPy}
libraries~\cite{SnapPy}, takes as an input any hyperbolic
manifold $M$ (with an arbitrary number of cusps), and gives as an output a list of candidate
exceptional fillings on $M$, including all truly exceptional ones.
The fact that indeed all truly exceptional fillings are included follows from the
use of the \emph{hikmot} library~\cite{HIKMOT}, that allows one to make SnapPy's numerical calculations rigorous.
For $M_5$, the list of candidate exceptional fillings actually turned out to be extremely small up to
the action of some symmetry groups, and the conclusion of the proof of
Theorem~\ref{main:intro:teo} was then obtained by explicitly recognizing the manifolds
resulting from the candidate exceptional fillings, thus proving them not to be hyperbolic; this process was carried
out both by hand and using the Matveev-Tarkaev~\cite{Rec} nice
\emph{Recognizer} program. We point out that the code {\tt find\_exceptional\_fillings.py} can be used on any hyperbolic manifold.
We have tested the code on various manifolds and
found that the output list of candidate exceptional fillings is typically correct
(\emph{i.e.},~all the fillings in the list are indeed exceptional).
The proof of Theorem~\ref{main:intro:teo} contained
in Section~\ref{proof:section} below may be viewed as a tutorial introduction to the code.

One of the key features of Theorem~\ref{main:intro:teo} is in our opinion the
unexpected shortness of its statement. In fact, our results show that
very little information is needed to list all the
exceptional fillings on $M_5$, and also to give a very
precise description of all the corresponding filled manifolds.
This is due to the many symmetries possessed by
$M_5$ and by some hyperbolic manifolds obtained as Dehn fillings of $M_5$.
In fact, Theorem~\ref{main:intro:teo} shows that
\emph{only $7$ exceptional fillings on $M_5$ are responsible for all the other ones}
---and the geometric meaning of the maps appearing in the statement
will be made more precise soon.
In addition, \emph{a single} exceptional filling on $M_5$ ---the first one listed in Theorem~\ref{main:intro:teo}--- is
responsible for the vast majority of the other ones: the remaining $6$ exceptional fillings can be viewed as being very sporadic.
Since most of the manifolds in the Callahan-Hildebrand-Weeks cusped census~\cite{CaHiWe} can be
obtained as fillings of $M_5$,
we can conclude that most of the exceptional fillings on the manifolds in the census are determined by this single
exceptional filling on $M_5$. For instance, all the 10 exceptional surgeries on the figure-8 knot are consequences of it.

\medskip

\noindent
\textbf{A notable finite(?) sequence of cusped hyperbolic manifolds}\quad
Figure~\ref{sequence:fig} shows five links in the $3$-sphere,
at least the first three of which are famous ones:
the \emph{figure-eight knot} K4a1$\,=4_1$,
the \emph{Whitehead link} L5a1\,=$\,5^2_1$, and the \emph{chain
link} with 3 components L6a5\,$=6^3_1$; we then have the chain L8n7\,$=8^4_2$
link with 4 components, and (of course) the minimally
twisted $5$-chain link L10n113 (for the first four links we are indicating both Thistlewaite's and
Rolfsen's names). Let us now denote by $M_i$ the exterior of the
$i$-th link in Fig.~\ref{sequence:fig}  (the notation is consistent for $M_5$). It is well-known that
each $M_i$ is an $i$-cusped hyperbolic manifold, and the finite sequence
$\left(M_i\right)_{i=1}^5$ has several interesting features.
We first note that $M_i$ is a filling of $M_{i+1}$ for all $i\leqslant 4$. We then recall that
each $M_i$ is conjectured~\cite{Ago} to have the smallest volume among
$i$-cusped hyperbolic manifolds; the conjecture was proved in~\cite{CaMe} for $i=1$ and in~\cite{Ago}
for $i=2$, and it is open for $i=3,4,5$. But the most remarkable properties of
$\left(M_i\right)_{i=1}^5$ arise when one considers their exceptional fillings.
The manifold $M_3$ was already called the \emph{magic} one in~\cite{GoLu, GoWu}, because it has
many notable exceptional fillings. The exceptional fillings on $M_1$ were classified by
Thurston~\cite{Thu}, those on $M_2$ and $M_3$ were classified by Martelli and Petronio~\cite{MaPe},
and those on $M_4$ and $M_5$ are classified here.

\begin{figure}
 \begin{center}
\includegraphics[width = 12 cm]{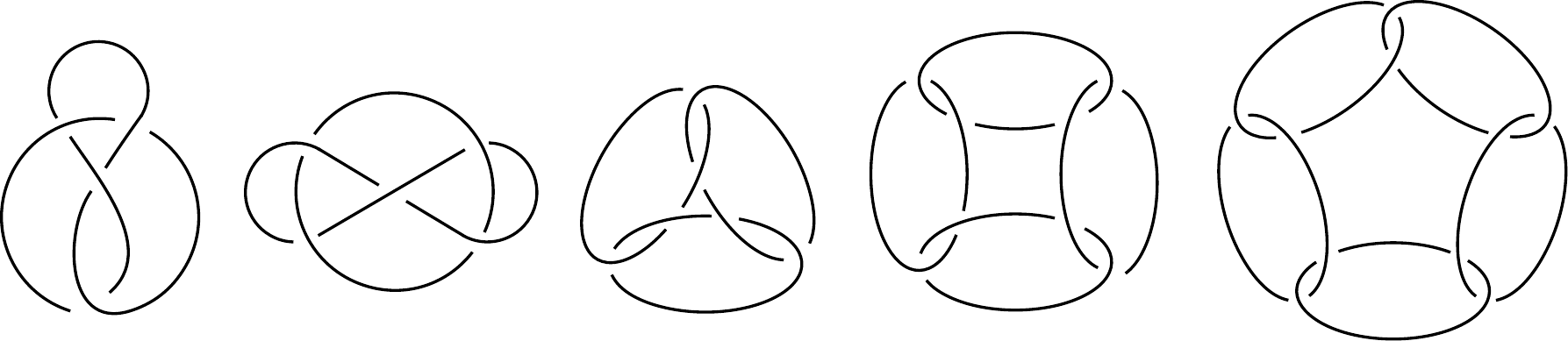}
 \end{center}
 \caption{The links in $S^3$ whose exteriors are denoted by $M_1$, $M_2$, $M_3$, $M_4$, and $M_5$.}
 \label{sequence:fig}
\end{figure}

Our next aim is to explain our discovery that \emph{the amount of information required
to describe the exceptional fillings on $M_i$ is roughly constant for $i=1,2,3,4,5$.}
Considering that the number of cusps and the volume of $M_i$ increase with $i$, we
believe that this is a rather remarkable fact.
However, to substantiate our statement we need to be a little more specific.
In fact, as soon as a hyperbolic manifold $M$ has more than one cusp, infinitely many fillings on $M$
are typically exceptional, but finitely many fillings are responsible for all other ones. To make this explicit,
we define an exceptional filling on $M$ to be \emph{isolated} if all its proper sub-fillings
(namely, those obtained by replacing at least one non-empty slope with an empty one)
are hyperbolic. A filling having an exceptional proper sub-filling is exceptional unless a very
special situation occurs, which is never the case for our $M_i$'s. We can then conclude that
\emph{a filling on $M_i$ is exceptional if and only if it contains an isolated exceptional filling}.

Getting to the actual data we have obtained, we start by
listing in Table~\ref{exceptional_1:table}
the numbers of isolated exceptional fillings on $M_i$ for $i=1,2,3,4,5$, split
according to the number $k$ of cusps filled.
\begin{table}
\begin{center}
\begin{tabular}{r||c|c|c|c|c||c}
& \multicolumn{5}{|c||}{Number of cusps filled} & \\ \cline{2-6}
Manifold & \begin{minipage}{.7cm}\begin{center}1\end{center}\end{minipage} &
    \begin{minipage}{.7cm}\begin{center}2\end{center}\end{minipage} &
        \begin{minipage}{.7cm}\begin{center}3\end{center}\end{minipage}     &
            \begin{minipage}{.7cm}\begin{center}4\end{center}\end{minipage} &
                \begin{minipage}{.7cm}\begin{center}5\end{center}\end{minipage} & Total\\ \hline\hline
$M_1$ & 10 &     &     &    & & 10\\
$M_2$ & 12 & 14  &    &     & & 26\\
$M_3$ & 15 & 15 & 52 & &  & 82\\
$M_4$ & 16 & 24 & 96 & 492 & & 628\\
$M_5$ & 15 & 30 & 180 & 780 & 5232 & 6237
\end{tabular}
\end{center}
\caption{Number of isolated exceptional fillings on $M_i$ according to the number of filled cusps.}
\label{exceptional_1:table}
\end{table}
As one sees, these numbers grow with $i$ and $k$, and the total number of isolated
exceptional fillings of $M_i$ appears to grow exponentially with $i$.
These numbers already reduce very considerably if we take into account the
action of the symmetry group of $M_i$, identifying exceptional fillings
equivalent under it, as shown in Table~\ref{exceptional_2:table}.
\begin{table}
\begin{center}
\begin{tabular}{r||c|c|c|c|c||c}
& \multicolumn{5}{|c||}{Number of cusps filled} & \\ \cline{2-6}
Manifold & \begin{minipage}{.7cm}\begin{center}1\end{center}\end{minipage} &
    \begin{minipage}{.7cm}\begin{center}2\end{center}\end{minipage} &
        \begin{minipage}{.7cm}\begin{center}3\end{center}\end{minipage}     &
            \begin{minipage}{.7cm}\begin{center}4\end{center}\end{minipage} &
                \begin{minipage}{.7cm}\begin{center}5\end{center}\end{minipage} & Total\\ \hline\hline
$M_1$ & 6    &     &     &    & & 6\\
$M_2$ & 6    & 8  &     &    & & 14\\
$M_3$ & 5    & 3  & 14 &  & & 22\\
$M_4$ & 2    & 2  &   4   &  22 & & 30\\
$M_5$ & 1    & 1  & 3  & 7 & 52 & 64
\end{tabular}
\end{center}
\caption{Number of isolated exceptional fillings on $M_i$ according to the number of filled cusps,
after identifying fillings obtained from each other under
the action of the symmetry group of $M_i$.}
\label{exceptional_2:table}
\end{table}
But the dramatic conclusion that \emph{the number of
``really inequivalent'' exceptional fillings of $M_i$ is
roughly constant for $i=1,2,3,4,5$} follows by taking into account another phenomenon.
In fact, it can and does happen that a hyperbolic manifold $N$ obtained by filling
some $M_i$ has symmetries that are not induced by symmetries of $M_i$.
If this situation, two isolated exceptional fillings on $N$ that are equivalent under
such symmetries of $N$ both contribute to the counting
in Table~\ref{exceptional_2:table}, but after all we can still identify them,
letting also the symmetries of $N$ act. Similarly, there can be an isolated exceptional
filling on $N$ that contributes to Table~\ref{exceptional_2:table} but is, under symmetries of
$N$, equivalent to a non-isolated filling on $M_i$, in which case we can disregard it.
By systematically taking into account the symmetries of the
hyperbolic fillings of the $M_i$'s we then get the figures of
Table~\ref{exceptional_3:table} (with the $7$ fillings on $M_5$ being precisely
those described in Theorem~\ref{main:intro:teo}).
\begin{table}
\begin{center}
\begin{tabular}{r||c|c|c|c|c||c}
& \multicolumn{5}{|c||}{Number of cusps filled} & \\ \cline{2-6}
Manifold & \begin{minipage}{.7cm}\begin{center}1\end{center}\end{minipage} &
    \begin{minipage}{.7cm}\begin{center}2\end{center}\end{minipage} &
        \begin{minipage}{.7cm}\begin{center}3\end{center}\end{minipage}     &
            \begin{minipage}{.7cm}\begin{center}4\end{center}\end{minipage} &
                \begin{minipage}{.7cm}\begin{center}5\end{center}\end{minipage} & Total\\ \hline\hline
$M_1$ & 6    &     &     &    & & 6 \\
$M_2$ & 6    & 2  &     &    & & 8 \\
$M_3$ & 5    & 1  & 3 &  & & 9\\
$M_4$ & 2    & 0  & 1 & 3 & & 6\\
$M_5$ & 1    & 0  & 0  & 1 & 5 & 7\\
\end{tabular}
\end{center}
\caption{Number of isolated exceptional fillings on $M_i$ according to the number of filled cusps,
after identifying fillings obtained from each other under
the action of the symmetry group of $M_i$ \emph{or} of hyperbolic manifolds obtained by filling $M_i$.}
\label{exceptional_3:table}
\end{table}
This proves that the
minimal number of exceptional fillings needed to generate
(via symmetries) all the exceptional fillings on $M_1$, $M_2$, $M_3$, $M_4$, and $M_5$ is indeed roughly constant.

\begin{table}
\begin{center}
\begin{tabular}{r||c|c|c|c}
Manifold & $M_2$ & $M_3$ & $M_4$ & $M_5$ \\
\hline
Computer time & $1''$ & $24''$ & $2'\,10''$ & $3'\,48''$
\end{tabular}
\end{center}
\caption{Computer time needed to classify the exceptional fillings.}
\label{computer:table}
\end{table}

We underline now that not only is the number of exceptional fillings
on $M_i$ for $i=1,2,3,4,5$ responsible for all other
exceptional fillings extremely small, but also that, once
the investigation is properly organized, the computer time
needed to detect them is very limited as well, see Table~\ref{computer:table}. Encouraged by these facts,
we believe that it should be possible to carry out a similar analysis
for manifolds having $6$ or more cusps, and we put forward two questions.
To state them, we define a sequence $\left(M_i\right)_{i=1}^{+\infty}$ to be
\emph{universal} if every compact 3-manifold with (possibly empty)
boundary consisting of tori can be obtained by Dehn filling on some $M_i$.
We then have the following (slightly vague):

\begin{quest}
Is there a universal sequence $\left(M_i\right)_{i=1}^{+\infty}$ with
$M_i$ an $i$-cusped hyperbolic 3-manifold such that
the exceptional fillings on each $M_i$ can be described using a small amount of data?
\end{quest}

A much more ambitious question is the following:

\begin{quest}
Is there a universal sequence $\left(M_i\right)_{i=1}^{+\infty}$ with
$M_i$ an $i$-cusped hyperbolic 3-manifold such that
the exceptional fillings on all the $M_i$'s can be simultaneously described using
a finite amount of data?
\end{quest}

\medskip

\noindent
\textbf{On the computer assisted-proof}\quad
In the original 2011 version of this paper, to rigorously verify the numerical computations made by SnapPy,
we used an algorithm constructed in~\cite{Mo, Mo:page} and implemented in Pari~\cite{Pari}.
It was then pointed out to us that this algorithm might not be 100\% immune from round-off errors. Motivated by another project \cite{IM}, Hoffman, Ichihara, Kashiwagi, Masai, Oishi, and Takayasu wrote in 2013 a different algorithm \cite{HIKMOT} that has at least three advantages: it uses interval arithmetic (and is hence immune from
round-off errors), it is very fast, and it is written directly in python 
(so we were able to incorporate it directly in our previous code). 
The 2013 versions of our paper and codes use this new algorithm.

\medskip

\noindent
\textsc{Acknowledgements} \
We warmly thank Neil Hoffman and Hidetoshi Masai for many helpful conversations, and the anonymous referees for pointing out a
mistake in our tables.

\section{Main results} \label{main:section}
We describe here in greater detail the hyperbolic manifold $M_5$, namely the exterior of the
\emph{minimally twisted chain link} with 5 components L10n113 shown in Fig.~\ref{chain5_fig:fig}.
In particular, we list its symmetries  and we
analyze some of its notable exceptional and hyperbolic fillings. We then state Theorem~\ref{main:teo},
that classifies all the exceptional fillings on $M_5$.

\subsection{Hyperbolic structure}
The reflection (or rotation of angle $\pi$) across
the dotted circle in Fig.~\ref{pentangle:fig}-left leaves the link invariant
and thus gives an involution $\iota$ on $M_5$. The quotient of $M_5$ under $\iota$ is an orbifold
whose singular set consists of $10$ order-2 edges, bounded by $5$ spheres with $4$ order-$2$ cone points,
each of which is the quotient of a toric component of $\partial M_5$ under the elliptic involution.
Collapsing these spheres to points we get an orbifold
with total space $S^3$ and singular set the \emph{pentangle} graph
shown in Fig.~\ref{pentangle:fig}-right. This implies that the interior of $M_5$
is the double branched cover of $S^3$ minus the vertices of the pentangle, branched
along the edges of the pentangle.
\begin{figure}
 \begin{center}
\includegraphics[width = 12 cm]{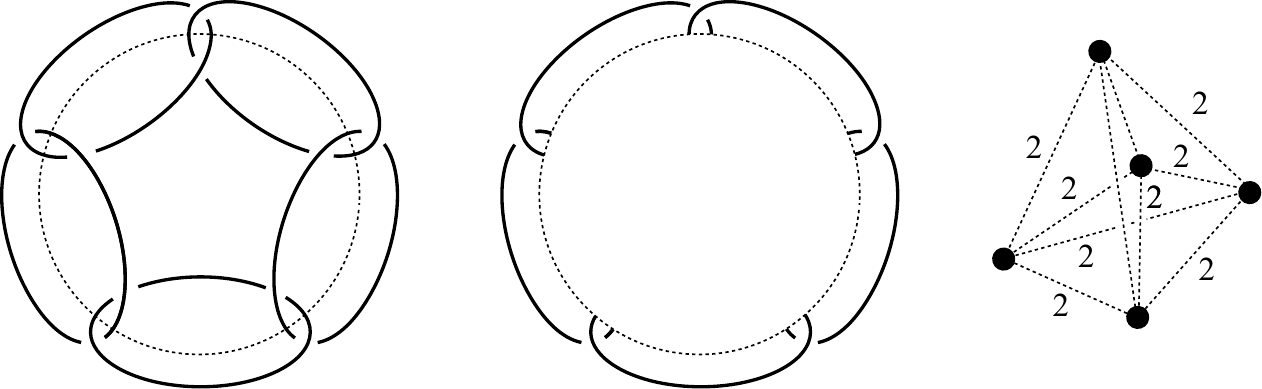}
 \end{center}
 \caption{Left: a circle the reflection across which leaves the link invariant.
 Center: the corresponding quotient graph.
 Right: the pentangle graph, obtained from the quotient graph by contracting each solid arc to a vertex.}
 \label{pentangle:fig}
\end{figure}

We will now construct the hyperbolic structure on (the interior of) $M_5$
as the double cover of the hyperbolic structure on the pentangle orbifold
(minus the vertices), as described in~\cite{DuTh}. To this end,
consider $S^3$ to be the boundary of the 4-dimensional simplex and the pentangle graph to be
the 1-skeleton of this 4-simplex. Next, realize each of the $5$ codimension $1$-faces of the
$4$-simplex as a regular ideal hyperbolic tetrahedron in $\mathbb H^3$, and glue the faces
of these ideal tetrahedra using
isometries. Each ideal tetrahedron has dihedral angle $\frac\pi3$ along each of its edges, and
every edge of the pentangle graph is adjacent to three ideal tetrahedra, so the cone angle
along each edge of the pentangle graph is $3\times \frac{\pi}{3} = \pi$.
The link of each vertex of the pentangle consists of $4$ Euclidean equilateral triangles that
glue nicely to give a Euclidean structure with $4$ cone points of angle $\pi$ on the $2$-sphere,
so we get an orbifold hyperbolic structure on the pentangle graph.
Pulling back this structure to $M_5$ we see that its hyperbolic structure is obtained by
gluing together 10 regular ideal hyperbolic tetrahedra.
In particular, the volume of $M_5$ is equal to $10 \times v_3= 10.149416 \ldots$

\subsection{Symmetries and slopes}
Every permutation of the vertices of the pentangle graph is realized by a unique isometry of the pentangle
orbifold, whose symmetry group is therefore $\permu_5$. Every isometry of the quotient
lifts to $M_5$, whose symmetry group is seen to be
isomorphic to $\permu_5 \times \matZ/\!_2$, with the factor $\matZ/\!_2$
generated by the involution $\iota$ and $\permu_5$ permuting the cusps, as one can check using SnapPy~\cite{SnapPy}.
Note that the symmetry group of $M_5$ is larger than the symmetry group of L10n113: the latter group is isomorphic to
$D_5 \times \matZ/\!_2$, with $\matZ/\!_2$ again generated by $\iota$ and $D_5$ being the order-$10$ dihedral group,
generated by a rotation of angle $\frac{2\pi}5$ around an axis orthogonal to the projection plane
in Fig.~\ref{chain5_fig:fig}, and by the reflection across a suitable plane containing this axis.

Let us now number from $1$ to $5$ the components of L10n113, in such a way
that the $i$-th component is linked with the $(i+1)$-th (there is a $D_5$ ambiguity for doing this,
that we will soon view to be immaterial). Correspondingly, we have a numbering from $1$ to $5$
of the components of $\partial M_5$.
We also fix a meridian-longitude oriented homology basis $(\mu_i,\lambda_i)$ on the $i$-th
component of $\partial M_5$, viewed
as the boundary of the exterior of a (trivial) knot in $S^3$. This basis is defined up to simultaneous sign
reversal, so if a slope represents $\pm(p_i\mu_i+q_i\lambda_i)$ in homology we have an
element $p_i/q_i$ of $\matQ \cup \{\infty\}$ uniquely defined by the slope, and conversely.
This shows that a filling on one boundary component is described by an element of
$\Phi=\matQ \cup \{\infty,\emptyset\}$. Note that $\infty$ is the meridian and $0$ is the longitude.
Using the numbering we then see that a filling
on $M_5$ is described by a $5$-tuple $(\alpha_1,\alpha_2,\alpha_3,\alpha_4,\alpha_5)\in\Phi^5$.
The corresponding filled manifold will be denoted by
$M_5(\alpha_1, \alpha_2,\alpha_3,\alpha_4, \alpha_5).$

Every symmetry of $M_5$ sends a slope on a component of $\partial M_5$ to some (other) slope on some
(other) component of $\partial M_5$, so we have an action of $\permu_5$ on $\Phi^5$ (one easily sees that
$\iota$ acts trivially on all slopes, so we dismiss it). To describe this action we start with the easy part coming from
the (dihedral) symmetries of the link: each such symmetry
send meridians to meridians and longitudes to longitudes, therefore the action of $D_5$ on $\Phi^5$
is generated by the maps
\begin{align}
(\alpha_1, \alpha_2,\alpha_3,\alpha_4,\alpha_5) & \longmapsto (\alpha_5, \alpha_1, \alpha_2, \alpha_3, \alpha_4) \label{first:eqn} \\
(\alpha_1, \alpha_2, \alpha_3, \alpha_4, \alpha_5) & \longmapsto (\alpha_5, \alpha_4, \alpha_3, \alpha_2, \alpha_1).\label{firstbis:eqn}
\end{align}
This implies in particular that the above-chosen ordering of the components of $\partial M$ is irrelevant.
To generate the full action of $\permu_5$ on $\Phi^5$ it is now sufficient to add the action on slopes
of a symmetry that switches two boundary components leaving the other three invariant. Using
SnapPy one sees that one such map is as follows:
\begin{equation} \label{transposition:eqn}
(\alpha_1, \alpha_2,\alpha_3,\alpha_4,\alpha_5) \longmapsto
\left(\tfrac1{\alpha_2}, \tfrac1{\alpha_1},1-\alpha_3,\tfrac{\alpha_4}{\alpha_4-1},1-\alpha_5\right).
\end{equation}
This is an orientation-reversing isometry.
Note that the action on slopes on the invariant components is a non-trivial one; moreover
any of the slopes in the argument of the map is allowed to be $\emptyset$, in which case
there is a corresponding $\emptyset$ slope in the value.

\subsection{Notation for some graph manifolds}
\label{graph:subsection}
Let $\Sigma$ be an oriented surface, possibly with boundary. Given $k\in\matN$ and coprime pairs of integers
$(p_i,q_i)$ for $i=1,\ldots,k$ with $p_i\neq 0$ for all $i$, we denote by
$$\big(\Sigma, (p_1,q_1), \ldots, (p_k,q_k)\big)$$
the Dehn filled manifold $(\Sigma'\times S^1)(p_1\mu_1+q_1\lambda_1,\ldots,p_k\mu_k+q_k\lambda_k)$, where
$\Sigma'$ is $\Sigma$ with $k$ open discs removed with the induced orientation,
$\mu_i$ is the oriented component of $\partial\Sigma'$ corresponding to
the $i$-th disc removed, and $\lambda_i$ is the oriented $S^1$ on the same torus.
The result is a Seifert manifold with an exceptional fibre for each $i$ such that $|p_i|\geqslant 2$.
As an example of this construction we note that the Poincar\'e homology sphere can be described as
$$\seiftre{S^2}2{-1}3151.$$
If we allow one $p_i$ to be $0$, say $p_1=0$ and $q_1=1$, then $(\Sigma,(p_1,q_1),\ldots,(p_k,q_k))$
turns out to be the connected sum of the lens spaces $\lens{p_2}{q_2},\ldots,\lens{p_k}{q_k}$ and of
$2g$ copies of $S^2\times S^1$ if the genus of $\Sigma$ is $g$, with $h$ unlinked unknots removed
if $\partial\Sigma$ has $h$ components.

Our notation to encode (some) Seifert manifolds can now be promoted to encode (some)
graph manifolds. In fact, note that the boundary of $\big(\Sigma, (p_1,q_1), \ldots, (p_k,q_k)\big)$ is
given by $(\partial\Sigma)\times S^1$, so it consists of tori, and that on each of them
there is a preferred homology basis given by an oriented component of $\partial\Sigma$ and the
oriented $S^1$. Moreover any two boundary components are mapped to each other by a homeomorphism
of the manifold that matches the homology bases.
Given two such manifolds $M$ and $M'$ and a matrix $X \in \GL(2,\matZ)$ we can therefore
define without ambiguity the gluing $M\bigcup_XM'$, along a homeomorphism from
a boundary component of $M$ to one of $M'$ whose action on homology
is expressed by $X$ with respect to the given bases. Note that $M$ and $M'$ are oriented, so
$M\bigcup_XM'$ is naturally oriented if $\det(X)=-1$, while it is merely orientable
if $\det(X)=+1$. But we can always reverse the orientation of $M$, which corresponds to changing each $(p_i,q_i)$ to
$(p_i,-q_i)$, to get $\det(X)=-1$, that we will always do for aesthetic reasons even if
we only care about orientable but unoriented manifolds. In a similar way one can define
a manifold $M/_X$ by gluing together along $X \in \GL(2,\matZ)$ two boundary components of
the same $M$, but in this case one \emph{must} have $\det(X)=-1$ to get an orientable result.
Denoting by $D$ and $A$ the disc and the annulus, respectively, we can consider
for instance the following closed graph manifolds:
$$ \seifdue{D}212{-1} \bigu 1514 \seifdue{D}2132, \qquad \seifuno A 2{1} \bigb 1211.$$
Note that if a gluing of two Seifert manifolds (or a gluing of a Seifert manifold to itself) is
performed along $X\in \GL(2,\matZ)$, then the absolute value of the top-right entry in $X$
has an instrinsic meaning, since it represents
the geometric intersection number on the gluing torus
between the two Seifert fibres coming from opposite sides of the torus.

\subsection{Some notable exceptional fillings} \label{notable:subsection}
We describe here some exceptional fillings on $M_5$. Theorem~\ref{main:teo} will then assert
that these are the only ones up to the maps (\ref{first:eqn})-(\ref{transposition:eqn})
induced by the symmetries of $M_5$ and up to further maps described below coming from symmetries
of hyperbolic fillings of $M_5$. In the sequel
for $k<5$ and $\alpha_1,\ldots,\alpha_k\in\Phi\setminus\{\emptyset\}$ we will interpret
$(\alpha_1,\ldots,\alpha_k)$ as the filling $(\alpha_1,\ldots,\alpha_k,\emptyset,\ldots,\emptyset)\in\Phi^5$ on $M_5$.

To begin, we note that filling the exterior of a
link with a slope $\infty$ corresponds to canceling the link component corresponding to the filled
boundary component of the exterior.
Therefore $M_5(\infty)$ is the exterior of the \emph{open} chain link
with 4 components shown in Fig.~\ref{open_chain:fig}; denoting by $P$ the pair-of-pants one then easily sees
that $M_5(\infty)$ is the graph manifold
$$F = (P \times S^1)\bigu 0110 (P\times S^1).$$
Applying the map (\ref{transposition:eqn}) with $\alpha_1=\infty$ and the other $\alpha_i$'s empty,
and then with $\alpha_4=\infty$ and the other $\alpha_i$'s empty, we see that the slopes $0$ and $1$
are equivalent to $\infty$ up to the symmetries of $M_5$, therefore $M_5(0)=M_5(1)=F$.
More exactly we already have $3$ exceptional slopes on each component of $\partial M_5$, for a total
of $15$, giving $F$ as a filling, and one easily sees that no other slope is obtained from them
under the maps (\ref{first:eqn})-(\ref{transposition:eqn}).
\begin{figure}
 \begin{center}
\includegraphics[width = 7 cm]{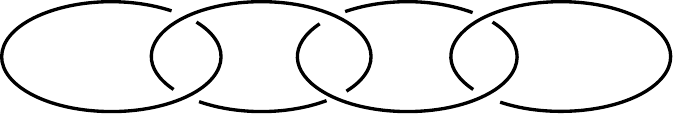}
 \end{center}
 \caption{An $\infty$ filling gives the exterior of an open chain with $4$ components
 (the connected sum of three copies of the Hopf link L2a1\,$=2^2_1$).}
 \label{open_chain:fig}
\end{figure}

\begin{rem}
If $\alpha\in\Phi^5$ contains some slope $0$, $1$, or $\infty$, then
$M_5(\alpha)$ is a filling of $F$, and a considerable variety of different
filled manifolds can already be obtained. The closed ones, for instance, have the form
$$ \seifdue D abcd \bigu 0110 \seifdue D efgh$$
for arbitrary filling coefficients $a,\ldots,h$.
If $|a|,|c|,|e|,|g|$ are all at least $2$ then this
is an irreducible 3-manifold whose JSJ decomposition
consists of two blocks and is transparent from the notation.
Allowing some of $a,c,e,g$ to be $0$ or $\pm1$ we also
get all small Seifert spaces and all reducible manifolds of
the form $\lens pq\# \lens rs$.
\end{rem}

As we will see, the single exceptional filling $\infty$
is responsible for the vast majority of the exceptional fillings on $M_5$.
There are however a few sporadic cases of an independent nature, that we will now describe, to do
which we will first study some notable hyperbolic fillings of $M_5$, starting
from $M_5(-1)$. To understand $M_5(-1)$ and its fillings we recall that a filling of a link exterior
can be described diagrammatically by attaching a symbol in $\Phi$ to each link component.
\begin{figure}
\begin{center}
\includegraphics[width = 7 cm]{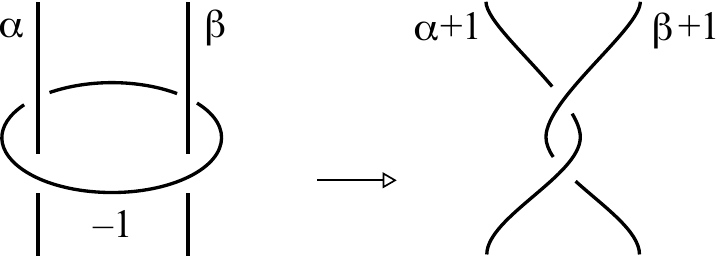}
\end{center}
\caption{The blow-down: a Fenn-Rourke move on surgery diagrams. On the left, the two strands
piercing the spanning disc of the unknot are supposed to be parts of two distinct components of the link.}
\label{Rolfsen:fig}
\end{figure}
Attaching a $-1$ to a component of L10n113
and applying the Fenn-Rourke move described in Fig.~\ref{Rolfsen:fig} (and
called blow-down in the sequel) we see that $M_5(-1)$ is actually the exterior $M_4$ of the
$4$-chain link L8n7 shown above in Fig.~\ref{sequence:fig}.
Using SnapPy one then sees that $M_4$ is a hyperbolic manifold obtained by suitably pairing
the faces of two regular ideal octahedra in $\matH^3$, therefore
it has volume $2\times 3.66386238\ldots = 7.32772475 \ldots$
The blow-down of Fig.~\ref{Rolfsen:fig} also shows that
$$M_5(-1,\alpha_2,\ldots,\alpha_5)=M_4(\alpha_2+1,\alpha_3,\alpha_4,\alpha_5+1)$$
after fixing a cyclic ordering
of the components of L8n7.
More precisely, the slopes on
$\partial M_5(-1)$ represented by $\alpha_2,\alpha_3,\alpha_4,\alpha_5\in\Phi$
using the homology bases coming from L10n113 are the slopes on
$\partial M_4$ represented by $\alpha_2+1,\alpha_3,\alpha_4,\alpha_5+1\in\Phi$
using the bases coming from L8n7. The link L8n7 has an obvious order-$4$
symmetry (a rotation of angle $\frac\pi2$ around an axis orthogonal to the projection plane)
giving a symmetry of $M_4$
that is not induced by a symmetry of $M_5$, and Fig.~\ref{fillings3:fig} shows
how this symmetry can be translated into the map
\begin{align} \label{blow:eqn}
\left(-1, \alpha_2,\alpha_3,\alpha_4,\alpha_5\right) & \longmapsto
\left(-1, \alpha_3-1,\alpha_4,\alpha_5+1,\alpha_2\right)
\end{align}
acting on $\{-1\}\times \Phi^4$ viewed as a subset of the fillings $\Phi^5$ on $M_5$.
One easily sees that (\ref{blow:eqn}) indeed cannot be deduced from (\ref{first:eqn})-(\ref{transposition:eqn}).
We also note that the latter functions map $-1$ to $\frac12$ or $2$, so $15$ different fillings
on $M_5$ give $M_4$.

\begin{figure}
 \begin{center}
\includegraphics[width = 11 cm]{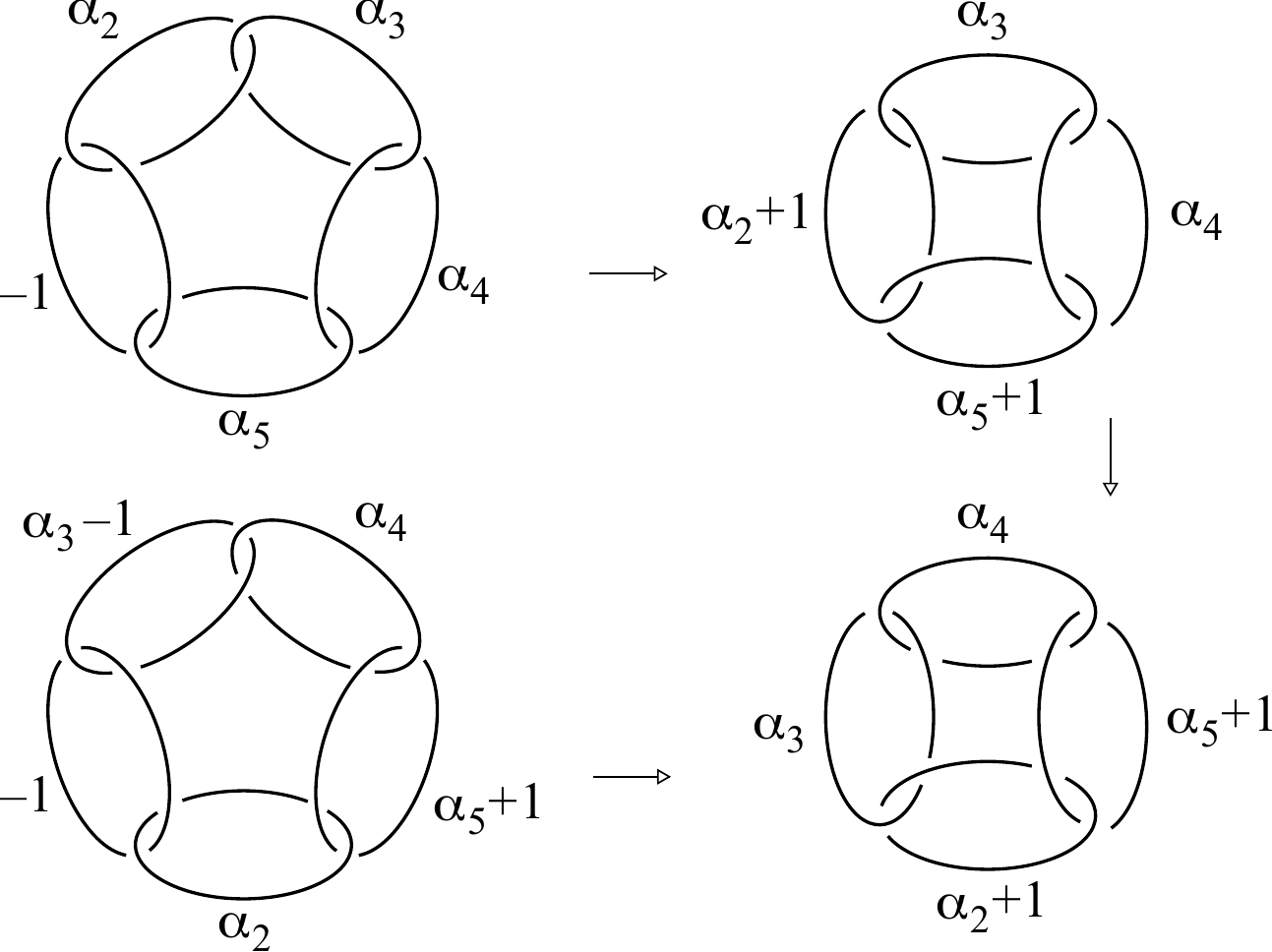}
 \end{center}
 \caption{Two distinct sets of slopes on the second to fifth component of $\partial M_5$ give
 on $M_4=M_5(-1)$ slopes that are obtained from each other by the order-$4$ symmetry of $M_4$.}
 \label{fillings3:fig}
\end{figure}

The rest of the sequence  of link exteriors $\left(M_i\right)_{i=1}^5$
introduced in Fig.~\ref{sequence:fig} is obtained
in a similar fashion, with $M_{i-1}=M_i(-1)$ as shown in Fig.~\ref{fillings2:fig}, whence
$$\begin{array}{ll}
M_4=M_5(-1)& M_3=M_5(-1,-2)\\
M_2=M_5(-1,-2,-2)&M_1=M_5(-1,-2,-2,-2)
\end{array}$$
(though several alternative realizations of $M_i$ as a filling of $M_5$ exist).
Of the links in Fig.~\ref{sequence:fig}, only the figure-eight knot is amphichiral, and its
equivalence with its mirror image induces on the fillings on $M_5$ the partial map
\begin{align}
\label{figure8:eqn}
\left(-1, -2, -2, -2, \alpha_5 \right) & \longmapsto \left(-1, -2, -2, -2, -\alpha_5 - 6 \right)
\end{align}
as one can check with SnapPy or prove using blow-downs as in Fig.~\ref{fillings3:fig}.
\begin{figure}
 \begin{center}
\includegraphics[width = 11.5 cm]{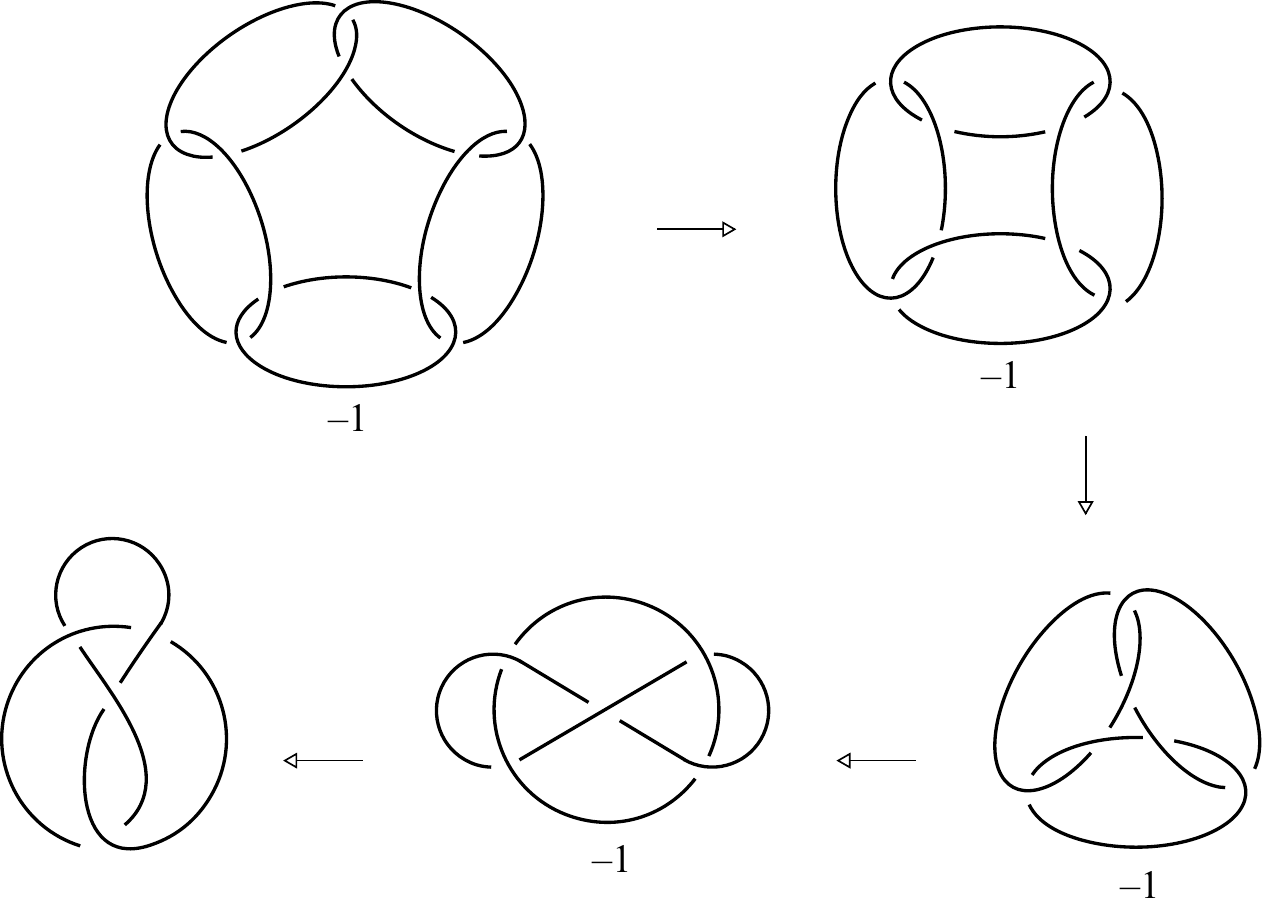}
 \end{center}
 \caption{Each $M_{i-1}$ as a $(-1)$-filling on $M_i$ for $i=2,3,4,5$}
 \label{fillings2:fig}
\end{figure}

We can now get back to the exceptional fillings on $M_5$, showing that for $i=2,3,4,5$ if all the
boundary components of $M_i$ are filled along the slope $-2$, as illustrated in Fig.~\ref{minus2:fig},
then the resulting manifold is non-hyperbolic.
\begin{figure}
 \begin{center}
\includegraphics[width = 12.5 cm]{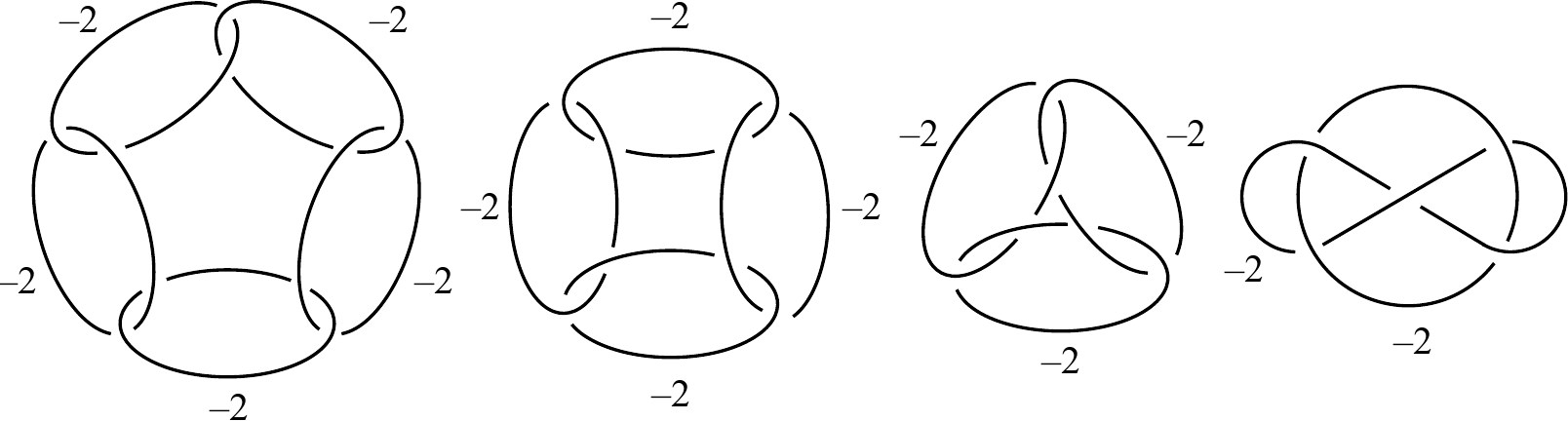}
 \end{center}
 \caption{Four exceptional fillings that cannot be obtained by filling $F$.}
 \label{minus2:fig}
\end{figure}
Indeed, using~\cite{MaPe} for $i=2,3$ and the recognizer~\cite{Rec} for $i=4,5$ we were able to prove that
$$M_i(-2,\ldots,-2)=\seifdue D212{-1} \bigu {-1}i1{-i+1} \seifdue D2131.$$
None of these manifolds can be obtained as a filling of $F=M_5(\infty)$,
since its JSJ decomposition is that apparent in its description,
and the geometric intersection number $i$ on the splitting torus of the
Seifert fibrations of the two JSJ blocks is greater than $1$, whereas it can only be $1$ for a graph-manifold filling of $F$.
Recalling that $M_{i-1}$ was recognized to be $M_i(-1)$ using the blow-down of
Fig.~\ref{Rolfsen:fig} it is now easy to see that the exceptional fillings just described
imply that the following elements of $\Phi^5$ are exceptional for $M_5$:
$$\begin{array}{cc}
(-1, -2, -2, -3, -5) & (-1, -2, -3, -2, -4) \\
(-1, -3, -2, -2, -3) & (-2, -2, -2, -2, -2).
\end{array}$$
We conclude with two more exceptional fillings, as described in
Fig.~\ref{last_exceptional:fig}, one on $M_2$ and one on $M_5$.
Using~\cite{MaPe} and the \emph{Recognizer}, the resulting
manifolds were identified to be respectively
$$M_2(0)=(P\times S^1) \bigb 0110 \qquad  M_5\left(-2, -\tfrac 12, 3, 3, -\tfrac 12\right)=\seifuno{A}{2}{-1} \bigb 1211.$$
Neither of them can be obtained as a filling of $F$, and the latter
cannot be obtained as a filling of the former,
because the geometric intersection numbers of the Seifert fibres on
the non-separating JSJ tori are different. Using the blow-down
of Fig.~\ref{Rolfsen:fig} to identify $M_2(0)$ as a filling of $M_5$, we get
the following further exceptional fillings on $M_5$:
$$(-1, -2, -2, -1) \qquad \left(-2, -\tfrac 12, 3, 3, -\tfrac 12\right).$$
\begin{figure}
 \begin{center}
\includegraphics[width = 8 cm]{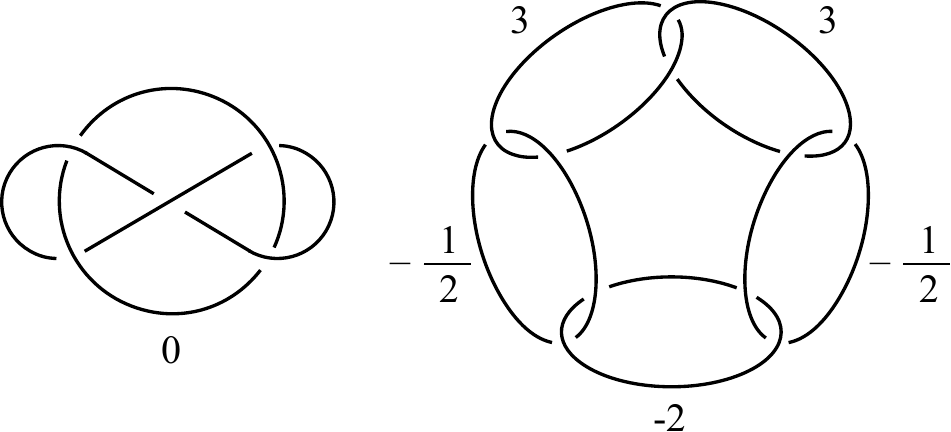}
 \end{center}
 \caption{Two exceptional fillings of $M_5$ not obtained by filling $F$.}
 \label{last_exceptional:fig}
\end{figure}

\subsection{The main theorem}
Quite surprisingly, the $7$ exceptional fillings on $M_5$
just described turn out to be sufficient to describe all other ones up
to the symmetries of $M_1$, $M_4$, and $M_5$:

\begin{teo} \label{main:teo}
Every exceptional filling on $M_5$ is equivalent up to a composition
of the maps \emph{(\ref{first:eqn})-(\ref{figure8:eqn})} to a filling containing one of
$$\begin{array}{rcr}
\infty \quad (-1, -2, -2, -1)  & \left(-2, -\tfrac 12, 3, 3, -\tfrac 12\right) & (-1, -2, -2, -3, -5)\\
(-1, -2, -3, -2, -4) & (-1, -3, -2, -2, -3) & (-2, -2, -2, -2, -2).
\end{array}$$
Moreover, no two of these seven fillings are related to each other by any composition of the
maps \emph{(\ref{first:eqn})-(\ref{figure8:eqn})}.
\end{teo}
The exact identification given above of the manifolds obtained by filling $M_5$ along the $7$ tuples of
slopes listed in Theorem~\ref{main:teo} now implies, together with the
discussion in Subsection~\ref{graph:subsection},
that a precise description can be provided
of \emph{all} the manifolds arising as exceptional fillings of $M_5$. We will now spell this out
in the closed case, \emph{i.e.}, for manifolds obtained by filling all 5 cusps of $M_5$, but
the analogous statement for the case where fewer cusps are filled could be easily given:

\begin{cor} \label{main:cor}
Every closed filling of $M_5$ is hyperbolic, except those
listed below and those obtained from them via compositions
of the maps \emph{(\ref{first:eqn})-(\ref{figure8:eqn})}:
\begin{align*}
M_5\left(\infty, \tfrac ab, \tfrac cd, \tfrac ef, \tfrac gh\right) & = \seifdue D a{-b}dc \bigu 0110 \seifdue D feg{-h}\\
M_5\left(-1,-2, -2, -1, \tfrac ab \right) & =  \seifuno{A}{b}{-a-b} \bigb 0110\\
M_5\left(-1, -2, -2, -3, -5\right) & = \seifdue{D}212{-1} \bigu {-1}21{-1} \seifdue{D}2131 \\
M_5\left(-1, -2, -3, -2, -4\right) & = \seifdue{D}212{-1} \bigu {-1}31{-2} \seifdue{D}2131 \\
M_5\left(-1, -3, -2, -2, -3\right) & = \seifdue{D}212{-1} \bigu {-1}41{-3} \seifdue{D}2131 \\
M_5\left(-2, -2, -2, -2, -2\right) & =\seifdue{D}212{-1} \bigu {-1}51{-4} \seifdue{D}2131 \\
M_5\left(-2, -\tfrac 12, 3, 3, -\tfrac 12\right) & = \seifuno A 2{-1} \bigb 1211.
\end{align*}
\end{cor}
\begin{proof}
We only need to check the correct expression for the filled manifolds on first two lines,
depending on the filling parameters. The expression on the first line is easily derived from
Fig.~\ref{open_chain:fig}, and that on the second line is taken from~\cite{MaPe}.
\end{proof}

\section{The code}
We have written a python code, named
\begin{center}
{\tt find\_exceptional\_fillings.py}
\end{center}
available from~\cite{M} and based on the
SnapPy libraries written by Culler, Dunfield, and Weeks~\cite{SnapPy}.
The code takes as an input a manifold $N$ and gives as an output a list of candidate
isolated exceptional fillings of $N$. It uses the \emph{hikmot} python 
library \cite{HIKMOT} and contains a certain \emph{ad hoc} rewriting of some SnapPea routines.

We explain in this section how and why the code works; the reader who is interested only in its implementation may skip this section and read the next one.

\subsection{The algorithm}
The code follows a standard iterative algorithm,
already used in~\cite{MaPe}, that we briefly recall here.

For every cusp $T$ of $N$, the code uses SnapPy to determine the cusp shape $x+iy$ of $T$,
in the sense that up to dilation $T$ is the quotient of $\matC$ under the
lattice generated by $1$ and by $x+iy$.
The code also finds the area $A$ of $T$ in a maximal horospherical cusp section, where maximality is
meant in the sense that all cusps of $N$ are simultaneously expanded at equal volume until
they first cease to be embedded and disjoint. The code then
enumerates the finitely many slopes on $T$ having length at most 6.
Note that a slope expressed by $\frac pq$ with respect to the basis of $H_1(T)$
corresponding to the generators $1$ and $x+iy$ of the lattice giving $T$ as a quotient has length
\begin{equation} \label{pq:eqn}
\ell\left(\tfrac pq\right) = \sqrt{\tfrac Ay\big((p+xq)^2+(yq)^2\big)}.
\end{equation}
For every slope $s$ in this finite list, the code checks whether the filled
manifold $N(s)$ is hyperbolic or not. If it is, the algorithm proceeds
iteratively with $N(s)$. If it is not, it appends the slope $s$ to the
list of isolated exceptional fillings on $N$.
Every isolated exceptional filling on $N$ is guaranteed to be detected by this algorithm
thanks to the Agol-Lackenby ``6-theorem''~\cite{Ago6,L}, according to which a
filling on $N$ with all slopes all having length
greater than 6 is \emph{hyperbolike}, and hence hyperbolic thanks to geometrization.
(Perelman's general proof of the geometrization conjecture
is not strictly necessary if $N$ is a hyperbolic filling of $M_5$,
because the symmetry $\iota$ of $M_5$ acts on each cusp as the elliptic involution
and hence extends to every filling of $M_5$, therefore the orbifold theorem~\cite{Bo-Po}
ensures geometrization for all the fillings of $M_5$.)

The process of listing the slopes of length at most $6$ is of course
sensitive to any numerical approximation in the computation
of the cusp shape $x+iy$ and the area $A$. To ensure rigor,
our code computes these values up to some (very) small error (see below), thus guaranteeing that \emph{all} slopes of length at most $6$ are picked. Occasionally some slope of length more than $6$ might be included, causing  a little redundant
analysis but not affecting the accuracy of the proof.
It is also important to note that the program only produces a list of \emph{candidate} isolated exceptional
fillings, that one needs to verify by hand. More precisely, the code produces \emph{two} lists of
fillings of $N$, and each list requires some kind of \emph{a posteriori} confirmation. To describe these
lists we need to briefly recall how SnapPy constructs hyperbolic structures on 3-manifolds.

\subsection{Thurston's hyperbolicity equations}
We describe here the classical general strategy used to construct hyperbolic structures on 3-manifolds. Given a cusped manifold $N$ and some filling $s$ on it,
Thurston's method~\cite{Thu} to construct a hyperbolic structure on $N(s)$ consists in
taking one complex variable $z_i$ for each
tetrahedron in an ideal triangulation
of $N$, and trying to solve certain holomorphic equations in these variables.
The number of equations can be reduced to be equal to the number of variables,
and a hyperbolic structure on $N(s)$ is guaranteed in presence of
a \emph{geometric solution}, namely one
with $\Im(z_i)>0$ for all $i$. If it exists, a geometric solution is unique.
A solution with some zero or negative $\Im(z_i)$ corresponds (respectively) to some
tetrahedra becoming \emph{flat} or
\emph{negatively oriented}, and the manifold $N(s)$ is not guaranteed to be hyperbolic if such a
solution is found. Experimentally, if there are only a few tetrahedra having $\Im(z_i) \leqslant 0$
(for instance, if there is only one), then often $N(s)$ has a hyperbolic structure anyway. Sometimes
a geometric solution for the same manifold can be found by randomly modifying the triangulation.
Sometimes hyperbolicity can be established by passing to some finite cover, because if some cover of a manifold
is hyperbolic then the manifold also is. We have written~\cite{M} a short code named {\tt search\_geometric\_solutions.py}
that tries to guarantee hyperbolicity using these techniques if a geometric solution is
not found in the first place.

\subsection{Interval arithmetic}
SnapPy uses Newton's algorithm to find an approximate numerical solution of Thurston's hyperbolicity equations.
This algorithm is very efficient, but a numerical solution does not rigorously guarantee the
existence of a nearby exact solution. This annoying problem was recently solved in \cite{HIKMOT}: the authors wrote a python library called {\tt hikmot}, freely available from \cite{HIKMOT-web}, that contains the function {\tt verify\_hyperbolicity()}. The function takes a SnapPy triangulation as an input and tries to provide a rigorous geometric solution of Thurston's equations using the Krawczyk Test. If it succeeds, the function returns a geometric solution expressed via \emph{interval arithmetic}: in interval arithmetic every real number is replaced
by a small interval $[a,b]$ containing it, whose width $|a-b|$ is the unavoidable error. Various functions
and operations easily extend from real numbers to intervals, for instance $[a,b] + [c,d] = [a+c, b+d]$, so that error propagation is automatically kept under control.
The function {\tt verify\_hyperbolicity()} is of course not guaranteed \emph{a priori} to find a geometric solution, but when it does it provides as an output some variables $z_i$ (actually, small rectangles in the upper half-plane)
that are guaranteed to contain a geometric solution ---in particular, if
{\tt verify\_hyperbolicity()} returns a solution then the 
existence of a hyperbolic structure is rigorously guaranteed.

As explained above, to make the proof rigorous it is vital to list \emph{all} the slopes
having length at most $6$ in a system of embedded and disjoint horospherical
cusp sections. Including a few longer slopes is not a problem, but none of length
at most $6$ must be missed. The length of a slope depends on the computation
of the cusp shape $x+iy$ and area $A$, that
SnapPy's original routines determine
numerically without keeping track of error propagation.
We were then forced to rewrite these routines in python, using interval arithmetic:
the new routines produce intervals $A$, $x$, and $y$ 
that are guaranteed to contain the true values of $A$, $x$, and $y$. 
The length $\ell(\frac pq)$  
of the slope $\frac pq$ is then also an interval
$\big[\ell_1(\frac pq), \ell_2(\frac pq)\big]$, found using the interval-arithmetic analogue of (\ref{pq:eqn}).
Selecting all the slopes $\frac pq$ such that
$\ell_1(\frac pq) \leqslant 6$ we then obtained a 
finite list rigorously 
guaranteed to contain all possibly exceptional slopes.

\subsection{Cusp area and shape}
Our routines that calculate $A$, $x$, and $y$ differ at some points from those used by SnapPea and we thus describe them. The first step is the construction of some (probably non-maximal) horospherical section of each cusp.
Recall that the hyperbolic manifold is described
by an ideal triangulation and by a geometric solution $(z_i)$ of Thurston's equations. Every $z_i$ is actually a small rectangle, but in what follows we will manipulate it as if it were a complex number, since all the operations we will use have their counterpart in interval arithmetic.

The ideal triangulation induces a triangulation of each toric boundary component.
On each complete (\emph{i.e.},~unfilled) boundary component the code picks an arbitrary oriented
edge and assigns to it the complex length 1.
The code then uses the $z_i$'s to extend the complex length
assignment to all the oriented edges of the torus. By construction, the complex lengths of one edge with opposite orientations are
opposite to each other, and the sums of the complex edge-lengths along the
boundary of each oriented triangle is $0$.
These complex lengths fully describe the Euclidean
structure of the torus up to similarity. In addition,
the real length of an edge is just the modulus of its complex length, therefore Heron's formula can be used to
compute the area of each triangle and hence of the full Euclidean torus.
We know that the complex lengths describe
an embedded horospherical cusp section provided that
the area is small enough, and for
theoretical reasons~\cite{SnapPy} the universal upper bound $\frac38\sqrt 3$ on the area of each cusp
is sufficient to ensure that all the cusp sections are embedded and
disjoint. To get the desired horospherical section the code then
rescales the Euclidean structure on each torus (which amounts to
multiplying the complex length of every edge on the torus by a certain positive constant)
until the area is $\frac38\sqrt 3$.

Having found the complex lengths giving the Euclidean structure on the torus,
the code can now compute the cusp shape $x+iy$ with respect to the fixed
meridian-longitude homology basis. To do this, following SnapPea, the code
describes two curves representing the meridian and the longitude as
normal curves with respect to the triangulation of the toric cusp, and
it computes the corresponding complex translation lengths
by developing the triangulation along the curve.
The complex shape $x+iy$ is then the ratio of the translation lengths
of the longitude and of the meridian.

We now turn to the computation of the area $A$ of a maximal horospherical cusp section.
As we will explain soon, this is best carried out when the triangulation of the
manifold is not an arbitrary one, but the canonical Epstein-Penner~\cite{EP}
Euclidean decomposition, or a subdivision of it, that our code  then
asks SnapPy to find. The triangulation that SnapPy gives as an answer is not strictly
guaranteed to be the canonical one because SnapPy's computations are not rigorous, but we can take this into account.
In very sporadic cases SnapPy is unable to canonize
the triangulation, and in these situations we just set $A = \frac38\sqrt 3$,
which is always fine because the existence of embedded disjoint cusp sections each having this area is guaranteed on any hyperbolic
manifold. In all other cases, the code works with a triangulation which is very likely to be the canonical one, and,
as detailed above, it starts by constructing a
section of area $\frac38\sqrt 3$ at each cusp.
These cusp sections lift to infinitely many horoballs in $\mathbb H^3$, and, if the triangulation is canonical, the minimal distance
between two  distinct such horoballs is realized along some edge of the ideal triangulation.
For each edge $e$ of the triangulation the code then computes the distance $d(e)$ between
the horoballs centered at the ends of $e$, to do which it picks one
ideal tetrahedron of which $e$ is an edge, and it employs a nice formula to be found in the SnapPea kernel~\cite{SnapPy}.
Supposing the tetrahedron has vertices $0,1,2,3$ and denoting by
$e(p,q)$ the edge with ends $p$ and $q$, and by $w(r;p,q)$ the boundary edge
on the link of $r$ with ends on $e(r,p)$ and $e(r,q)$, the formula reads
$$ d(e(p,q)) = -\frac 12 \log(L(w(p;r,q))\cdot L(w(p;s,q)) \cdot L(w(q;r,p))\cdot L(w(q;s,p)) )$$
where $\{p,q,r,s\}=\{0,1,2,3\}$ and $L$ denotes the real length of a boundary edge with respect to the Euclidean structure already found.
Taking the minimum $d$ of $d(e)$ over all the edges $e$ of the triangulation,
one theoretically knows that by
rescaling all the complex edge lengths by a factor $e^{\frac d2}$
the new minimal $d$ is then $0$, and,
if the triangulation is indeed canonical, this gives a maximal horospherical cusp
section with all cusps of area $A = e^d\cdot \frac38\sqrt 3$.

However, we are not 100\% sure that the triangulation is canonical.
To deal with this issue we use the fact that, for any triangulation, to guarantee that embedded and
disjoint horospherical cusp sections have been found, besides checking the condition $d>0$,
one must make sure that
for every ideal tetrahedron $T$ and at every vertex $v$ of $T$ the three complex
lengths of the boundary edges of the link of
$v$ determine a triangle entirely contained in $T$. To do so, we
represent $T$ in the upper half-space model of $\matH^3$ as in Fig.~\ref{height:fig}, with $v$ at $\infty$.
\begin{figure}
\begin{center}
\includegraphics[width = 12.5cm]{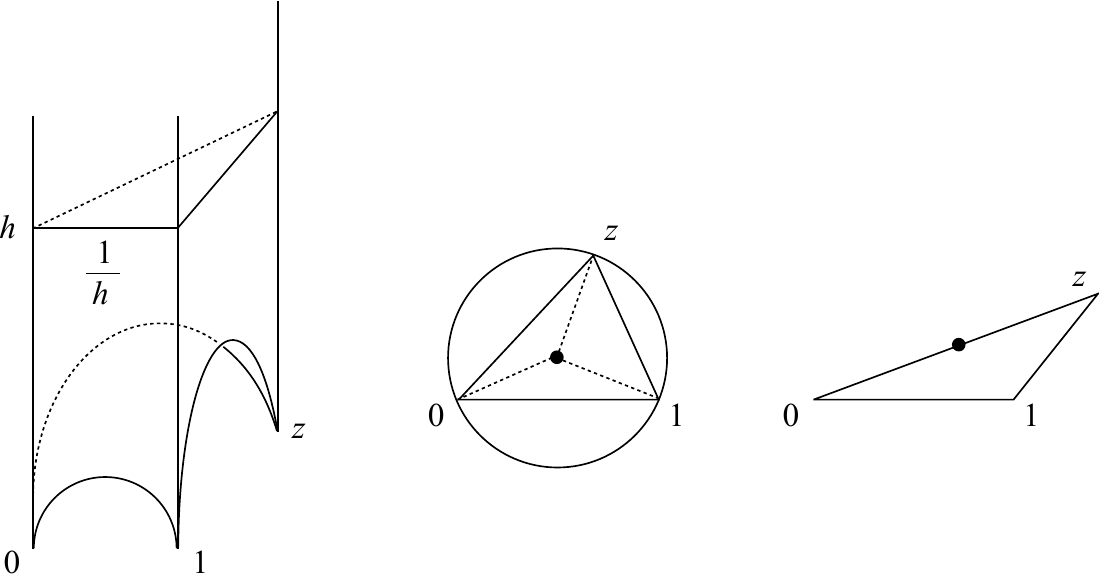}
 \end{center}
 \caption{The horospherical section is a triangle
 if its height $h$ is greater than some $k$. If all the inner angles of the triangle with vertices $0,1,z$ are
 acute, then $k$ is the radius of the circumcircle of the triangle, otherwise $k$ is
 half the length of the longest edge of the triangle.}
 \label{height:fig}
\end{figure}
The Euclidean height $h$ of the cusp section at $\infty$
is then easily seen to be the inverse of the length of the boundary
edge whose vertices lie
above $0$ and $1$, and the horizontal plane at height $h$ intersects $T$ in a triangle if and only if $h$ is greater than
some number $k$ that we now explain how to determine. Recall that $T$ is the intersection of four half-spaces, three
bounded by vertical planes and one
bounded by the hyperbolic plane whose circle at infinity $C$ contains $0, 1$, and $z$. If the center of $C$ lies in the triangle
of vertices $0$, $1$, and $z$ then $k$ equals the radius of $C$, otherwise $k$ equals
$\max\left\{\tfrac 12, \tfrac {|z|}2, \tfrac {|z-1|}2\right\}$.
Our code makes sure that embedded and disjoint horospherical cusp sections have been found
by checking that $h>k$ for every choice of $T$ and $v$, taking
errors into account. The code is actually designed to return $A = \frac38\sqrt 3$ as above
in case some test $h>k$ fails, but as a matter of fact this never happened
during our computations (which is not surprising, since the
triangulation $T$ is extremely likely to be the canonical one or a subdivision of it).

\section{Proof of the main theorem} \label{proof:section}
In this section we describe our proof of
Theorem~\ref{main:teo}, that, as anticipated above, is computer-assisted but rigorous.
Our main tool is the python code {\tt find\_exceptional\_fillings.py} described in the previous section and publicly available from \cite{M}.

\subsection{The output}
The code {\tt find\_exceptional\_fillings.py} takes a cusped hyperbolic manifold $N$
as an input and produces two lists as an output:
\begin{enumerate}[(I)]
\item A list of candidate exceptional fillings of $N$, for which SnapPy was unable to find any kind of solution;
\item A list of candidate hyperbolic fillings of $N$, for which SnapPy found some numerical non-geometric solution
(with flat or negatively oriented tetrahedra);
\end{enumerate}

Typically, list (II) consists of closed manifolds only. If this is the case,
the classification of the exceptional fillings of $N$ then
becomes complete and rigorous provided we can \emph{a posteriori} confirm the following:
\begin{enumerate}[(i)]
\item All the manifolds in list (I) are non-hyperbolic;
\item All the manifolds in list (II) are hyperbolic.
\end{enumerate}

On the contrary, when list (II) contains some manifold $Y$ with boundary tori, (ii) is necessary but not sufficient, because we should separately test for hyperbolicity the fillings of $Y$.  However, for all the $N$'s we had to deal with in the proof of Theorem \ref{main:teo}, list (II) indeed contained closed manifolds only.

\subsection{Organizing the search}
Despite our main objective being that to establish
Theorem~\ref{main:teo}, we have not run our code directly with $M_5$ as an input,
because it turned out that considerable computer time could be saved by dealing
first with the manifolds $M_2$, $M_3$ (thereby confirming the results of~\cite{MaPe}), and $M_4$, and
only then with $M_5$. In the next pages, along with the description of the data obtained step after step,
we will give an explanation of the exact strategy we have used to employ the data obtained at one step
to make the next step simpler.

\subsection{The Whitehead link exterior}
Running the code for $M_2$ (the exterior of the Whitehead link),
we get for this manifold the lists (I) and (II) described above.
With a slope $\frac pq\neq\emptyset$ written as \texttt{(p,q)},
the empty slope written as \texttt{(0,0)}, and pairs of slopes
(used to fill the two boundary components of $M_2$) written
between square brackets, the output is as follows:

\begin{verbatim}
Candidate exceptional fillings:
With 1 fillings:
[[(0,0),(0,1)],[(0,0),(1,0)],[(0,0),(1,1)],[(0,0),(2,1)],
[(0,0),(3,1)],[(0,0),(4,1)],[(0,1),(0,0)],[(1,0),(0,0)],
[(1,1),(0,0)],[(2,1),(0,0)],[(3,1),(0,0)],[(4,1),(0,0)]]
Total: 12
With 2 fillings:
[[(-4,1),(-1,1)],[(-3,1),(-1,1)],[(-2,1),(-2,1)],
[(-2,1),(-1,1)],[(-1,1),(-4,1)],[(-1,1),(-3,1)],
[(-1,1),(-2,1)],[(-1,1),(-1,1)],[(3,2),(5,1)],
[(4,3),(5,1)],[(5,1),(3,2)],[(5,1),(4,3)],
[(5,2),(7,2)],[(7,2),(5,2)]]
Total: 14
Candidate hyperbolic fillings:
With 1 fillings:
[]
Total: 0
With 2 fillings:
[[(-3,1),(-2,1)],[(-2,1),(-3,1)],[(3,2),(6,1)],
[(5,2),(6,1)],[(6,1),(3,2)],[(6,1),(5,2)]]
Total: 6
\end{verbatim}

Getting back to our tasks, we first
simplify (i) ---the confirmation that the candidate exceptional fillings are indeed non-hyperbolic---
by noting that there is a symmetry of $M_2$ that switches the two boundary components, with trivial action
on slopes given by
\begin{align}
(\alpha_1, \alpha_2) & \mapsto (\alpha_2, \alpha_1).\label{W0:eqn}
\end{align}
Using (\ref{W0:eqn}) we see that to achieve (i)
we must only show that the fillings
$$0\quad 1\quad 2\quad 3\quad 4\quad \infty$$
on one component, and
$$\begin{array}{cccc}
(-2, -2) & (-1, -4) & (-1, -3) & (-1, -2)\\
(-1, -1) & \left( \tfrac 43, 5 \right)& \left( \tfrac 32, 5\right) & \left( \tfrac 52, \tfrac 72 \right)\alza
\end{array}$$
on both components are truly exceptional. Exceptionality of these fillings was already proved in~\cite{MaPe}, and
we can actually make the list shorter using the additional symmetries
\begin{align}
\left(-1, \alpha_2\right) & \mapsto \left(-1, -\alpha_2\right) \label{W1:eqn} \\
\left(5, \alpha_2\right) & \mapsto \left(5, \tfrac {\alpha_2}{\alpha_2-1}\right) \label{W2:eqn}
\end{align}
that follow from the amphichirality of $M_2(-1)$ and $M_2(5)$ ---the figure-eight knot exterior and its sibling--- and
are detected by SnapPy.
Under the action of these maps the set of exceptional fillings on $M_2$ reduces to
$$0\quad 1\quad 2\quad 3\quad 4\quad \infty\quad (-2,-2)\quad \left(\tfrac 52, \tfrac 72\right), $$
as summarized in Table~\ref{exceptional_2:table}.
Getting to task (ii) ---the confirmation that the candidate hyperbolic fillings are indeed hyperbolic---
up to (\ref{W0:eqn}) we must prove hyperbolicity only of
$$(-2, -3)\quad \left(\tfrac 32, 6\right)\quad \left( \tfrac 52, 6 \right), $$
which we can do using {\tt search\_geometric\_solution}. In all three cases the code
finds a geometric solution on some finite cover of degree at most $10$. Note
that $M_2\left(\tfrac32,6\right)$ is the closed hyperbolic manifold Vol3 analyzed in~\cite{JR},
for which no geometric solution is known.
We have thus rigorously proved the following:
\begin{teo} \label{Whitehead:teo}
Every exceptional Dehn filling on $M_2$ is equivalent up to a composition of the maps
\emph{(\ref{W0:eqn})-(\ref{W2:eqn})}
to a filling containing one of
$$0\quad 1\quad 2\quad 3\quad 4\quad \infty\quad (-2,-2)\quad \left(\tfrac 52, \tfrac 72\right).$$
Moreover, no two of these eight fillings are related to each other by any composition of the
maps \emph{(\ref{W0:eqn})-(\ref{W2:eqn})}.
\end{teo}

\begin{rem}
The last assertion of the previous theorem is readily proved by direct inspection,
and we will refrain from stating the analogous ones in the
subsequent Theorems~\ref{magic:teo} and~\ref{M4:teo}.
\end{rem}

\subsection{The magic manifold} \label{M3:subsection}
After the Whitehead link exterior, the next item in the sequence $\left(M_i\right)_{i=1}^5$
introduced in Subsection~\ref{notable:subsection} is the magic manifold $M_3$.
The exceptional fillings on $M_3$ were already classified in~\cite{MaPe}, and using our code we can
rigorously confirm this classification along the lines explained above.
Since $M_2$ is obtained by a $(-1)$-filling on any component of $M_3$
and the exceptional fillings on $M_2$ are those in Theorem~\ref{Whitehead:teo},
we can exclude from our search any filling on $M_3$ containing a $-1$.
In the code {\tt find\_exceptional\_fillings.py}
there is a list named {\tt exclude} that by default is empty and that for $M_3$ we set as
\begin{verbatim}
exclude = [[(-1,1),(0,0),(0,0)],[(0,0),(-1,1),(0,0)],
[(0,0),(0,0),(-1,1)]]
\end{verbatim}
then getting the following as an output:
\begin{verbatim}
Candidate exceptional fillings:
With 1 fillings:
[[(0,0),(0,0),(0,1)],[(0,0),(0,0),(1,0)],[(0,0),(0,0),(1,1)],
[(0,0),(0,0),(2,1)],[(0,0),(0,0),(3,1)],[(0,0),(0,1),(0,0)],
[(0,0),(1,0),(0,0)],[(0,0),(1,1),(0,0)],[(0,0),(2,1),(0,0)],
[(0,0),(3,1),(0,0)],[(0,1),(0,0),(0,0)],[(1,0),(0,0),(0,0)],
[(1,1),(0,0),(0,0)],[(2,1),(0,0),(0,0)],[(3,1),(0,0),(0,0)]]
Total: 15
With 2 fillings:
[[(0,0),(1,2),(4,1)],[(0,0),(3,2),(5,2)],[(0,0),(4,1),(1,2)],
[(0,0),(5,2),(3,2)],[(1,2),(0,0),(4,1)],[(1,2),(4,1),(0,0)],
[(3,2),(0,0),(5,2)],[(3,2),(5,2),(0,0)],[(4,1),(0,0),(1,2)],
[(4,1),(1,2),(0,0)],[(5,2),(0,0),(3,2)],[(5,2),(3,2),(0,0)]]
Total: 12
With 3 fillings:
[[(-2,1),(-2,1),(-2,1)],[(1,2),(5,1),(5,1)],
[(2,3),(4,1),(4,1)],[(3,2),(3,2),(4,1)],[(3,2),(3,2),(8,3)],
[(3,2),(4,1),(3,2)],[(3,2),(7,3),(7,3)],[(3,2),(8,3),(3,2)],
[(4,1),(2,3),(4,1)],[(4,1),(3,2),(3,2)], [(4,1),(4,1),(2,3)],
[(4,3),(5,2),(5,2)],[(5,1),(1,2),(5,1)],[(5,1),(5,1),(1,2)],
[(5,2),(4,3),(5,2)],[(5,2),(5,2),(4,3)],[(5,2),(5,3),(5,3)],
[(5,3),(5,2),(5,3)],[(5,3),(5,3),(5,2)],[(7,3),(3,2),(7,3)],
[(7,3),(7,3),(3,2)],[(8,3),(3,2),(3,2)]]
Total: 22
Candidate hyperbolic fillings:
With 1 fillings:
[]
Total: 0
With 2 fillings:
[]
Total: 0
With 3 fillings:
[[(-3,1),(-2,1),(-2,1)],[(-3,1),(-2,1),(4,1)],
[(-2,1),(-3,1),(-2,1)],[(-2,1),(-2,1),(-3,1)],
[(-2,1),(1,2),(3,2)],[(2,3),(4,1),(-2,1)],
[(4,3),(3,2),(4,1)],[(4,3),(4,1),(3,2)],
[(5,1),(3,2),(3,2)],[(5,3),(3,2),(4,1)],
[(5,3),(4,1),(3,2)],[(7,3),(3,2),(3,2)],
[(7,3),(3,2),(5,3)]]
Total: 13
\end{verbatim}

To proceed we note that any element of the permutation group $\permu_3$ on the cusps of $M_3$
is realized by an isometry of $M_3$, with action on the slopes generated by two transpositions:
\begin{align}
(\alpha_1, \alpha_2, \alpha_3) & \mapsto (\alpha_3, \alpha_2, \alpha_1) \label{magic1:eqn} \\
(\alpha_1, \alpha_2, \alpha_3) & \mapsto (\alpha_2, \alpha_1, \alpha_3) \label{magic2:eqn}
\end{align}
In addition, as already stated in~\cite{MaPe}, a few fillings of $M_3$ have additional symmetries,
inducing the following partial maps on the slopes:
\begin{align}
\left(\tfrac 12, \alpha_2,\alpha_3\right) & \mapsto \left(\tfrac 12, 4-\alpha_2,4-\alpha_3\right) \label{M3_begin:eqn}\\
\left(\tfrac 32, \alpha_2,\alpha_3\right) & \mapsto \left(\tfrac 32, \tfrac {2\alpha_2 - 5}{\alpha_2 - 2},
    \tfrac {2\alpha_3 - 5}{\alpha_3 - 2}\right) \label{M3_a:eqn}\\
\left(\tfrac 52, \alpha_2,\alpha_3\right) & \mapsto \left(\tfrac 52, \tfrac {\alpha_2 - 3}{\alpha_2 - 2},
    \tfrac {2\alpha_3 - 3}{\alpha_3 - 1}\right)  \label{M3_b:eqn}\\
\left(4, \alpha_2,\alpha_3\right) & \mapsto \left(4, \tfrac {\alpha_2 - 2}{\alpha_2 - 1},
    \tfrac {\alpha_3 - 2}{\alpha_3 - 1} \right) \label{M3_c:eqn}\\
\left(-1,-2,\alpha_3\right) & \mapsto \left(-1, -2, -\alpha_3-2\right) \label{M3_1:eqn} \\
\left(-1, 4,\alpha_3\right) & \mapsto \left(-1, 4, \tfrac{1}{\alpha_3}\right). \label{M3_2:eqn}
\end{align}

Note that (\ref{M3_1:eqn}) and (\ref{M3_2:eqn}) follow from (\ref{W1:eqn}) and (\ref{W2:eqn}) because
$M_3(-1,\alpha_2,\alpha_3)=M_2(\alpha_2+1,\alpha_3+1)$, whereas (\ref{W2:eqn})
induces at the level of $M_3$ a map that is obtained (in a complicated fashion) by the
other maps listed.
We now embed the achievement of tasks (i)-(iii) for $M_3$ in the proof of
the following result, that confirms the main statement of~\cite{MaPe}
---with an overall change of sign because the chain link considered there is the mirror image
of the one considered here:

\begin{teo}\label{magic:teo}
Every exceptional Dehn filling on $M_3$ is equivalent up to a composition of the maps
\emph{(\ref{magic1:eqn})-(\ref{M3_2:eqn})}
to a filling containing one of
$$0\quad1\quad2\quad3\quad\infty\quad(-1,-1)\quad(-1, -3, -3)\quad(-2, -2, -2) \quad \left(\tfrac 23, 4, 4\right).$$
\end{teo}
\begin{proof}
All the candidate exceptional fillings on $M_3$ contained in the list (I) produced by our code
for $M_3$ had already been precisely recognized and hence proved to be indeed exceptional in~\cite{MaPe}, which
completes task (i) for $M_3$. Letting the maps (\ref{magic1:eqn})-(\ref{magic2:eqn}) act we reduce the list to
$$\begin{array}{ccccc}
0 & 1 & 2 & 3 & \infty\\
\left(\tfrac 12, 4\right) &
    \left( \tfrac 32, \tfrac 52 \right) &
        (-2, -2, -2) &
            \left(\tfrac 12, 5, 5\right) &
                \left( \tfrac 23, 4, 4 \right)\alza\\
\left( \tfrac 43, \tfrac 52, \tfrac 52 \right) &
    \left( \tfrac 32, \tfrac 73, \tfrac 73 \right) &
        \left( \tfrac 32, \tfrac 32, \tfrac 83 \right) &
            \left( \tfrac 32, \tfrac 32, 4 \right) &
                \left(  \tfrac 53, \tfrac 53, \tfrac 52 \right)\alza
\end{array}$$
to which we add
$$\begin{array}{cccc}
(-1, -1) & (-1, -3, -3) & (-1, -2, -5) & (-1, -2, -4)\\ (-1, -2, -3) & (-1, -2, -2) & \left( -1, \tfrac 13, 4 \right)
\end{array}$$
coming from the isolated exceptional fillings on $M_2 = M_3(-1)$ found in Theorem~\ref{Whitehead:teo},
under the identification $M_3(-1,\alpha_2,\alpha_3)=M_2(\alpha_2+1,\alpha_3+1)$ given by
the blow-down move of Fig.~\ref{Rolfsen:fig}. This is in complete agreement with~\cite{MaPe}, and
letting all the maps (\ref{M3_begin:eqn})-(\ref{M3_2:eqn}) act
we now get the list of $9$ exceptional fillings in the statement.

Turning to task (ii), the code {\tt search\_geometric\_solutions.py} succeeds in finding a geometric solution for each
candidate hyperbolic filling of list (II) for $M_3$.
\end{proof}

\subsection{The 4-chain link exterior} \label{M4:subsection}
Our next aim is to classify the exceptional fillings on the exterior
$M_4$ of the chain link with 4 components shown in Fig.~\ref{sequence:fig}.
The symmetry group of this link is isomorphic to $D_4\times \matZ/\!_2$, with the factor $\matZ/\!_2$ generated by an
involution similar to the $\iota$ described in Fig.~\ref{pentangle:fig}, with trivial action on fillings, and
two generators of the dihedral group $D_4$ act on fillings as follows:
\begin{align}
(\alpha_1, \alpha_2,\alpha_3,\alpha_4) & \longmapsto (\alpha_4, \alpha_1, \alpha_2, \alpha_3) \label{M4_begin:eqn} \\
(\alpha_1, \alpha_2, \alpha_3, \alpha_4) & \longmapsto (\alpha_4, \alpha_3, \alpha_2, \alpha_1). \label{M4_second:eqn}
\end{align}
In addition to these symmetries coming from the link, the symmetry group of $M_4$
contains a $\matZ/\!_2\times \matZ/\!_2$
subgroup leaving each cusp invariant, with two generators
acting as follows on slopes:
\begin{align}
\left(\alpha_1,\alpha_2,\alpha_3,\alpha_4\right) & \longmapsto
    \left(\tfrac{\alpha_1-2}{\alpha_1-1},\tfrac{\alpha_2-2}{\alpha_2-1},
    \tfrac{\alpha_3-2}{\alpha_3-1},\tfrac{\alpha_4-2}{\alpha_4-1}\right)\label{M4_third:eqn}\\
\left(\alpha_1,\alpha_2,\alpha_3,\alpha_4\right) & \longmapsto
    \left(2-\alpha_1, \tfrac{\alpha_2}{\alpha_2-1},2-\alpha_3,\tfrac{\alpha_4}{\alpha_4-1}\right).\label{M4_end:eqn}
\end{align}

Before running our code on $M_4$ we can now exclude the filling $-1$, that gives $M_3$, and
all the fillings obtained from $-1$ under compositions of the maps (\ref{M4_begin:eqn})-(\ref{M4_end:eqn}), which
give $-1$, $\frac32$, $3$, $\frac12$  on each cusp, for a total of $16$ slopes.
In {\tt find\_exceptional\_fillings.py} we then modify the list {\tt exclude} to
\begin{verbatim}
exclude = [
[(-1,1),(0,0),(0,0),(0,0)],[(3,2),(0,0),(0,0),(0,0)],
[(3,1),(0,0),(0,0),(0,0)],[(1,2),(0,0),(0,0),(0,0)],
[(0,0),(0,0),(0,0),(-1,1)],[(0,0),(0,0),(0,0),(3,2)],
[(0,0),(0,0),(0,0),(3,1)],[(0,0),(0,0),(0,0),(1,2)],
[(0,0),(-1,1),(0,0),(0,0)],[(0,0),(3,2),(0,0),(0,0)],
[(0,0),(3,1),(0,0),(0,0)],[(0,0),(1,2),(0,0),(0,0)],
[(0,0),(0,0),(-1,1),(0,0)],[(0,0),(0,0),(3,2),(0,0)],
[(0,0),(0,0),(3,1),(0,0)],[(0,0),(0,0),(1,2),(0,0)]]
\end{verbatim}
getting the following output:
\begin{verbatim}
Candidate exceptional fillings:
With 1 fillings:
[[(0,0),(0,0),(0,0),(0,1)],[(0,0),(0,0),(0,0),(1,0)],
[(0,0),(0,0),(0,0),(1,1)],[(0,0),(0,0),(0,0),(2,1)],
[(0,0),(0,0),(0,1),(0,0)],[(0,0),(0,0),(1,0),(0,0)],
[(0,0),(0,0),(1,1),(0,0)],[(0,0),(0,0),(2,1),(0,0)],
[(0,0),(0,1),(0,0),(0,0)],[(0,0),(1,0),(0,0),(0,0)],
[(0,0),(1,1),(0,0),(0,0)],[(0,0),(2,1),(0,0),(0,0)],
[(0,1),(0,0),(0,0),(0,0)],[(1,0),(0,0),(0,0),(0,0)],
[(1,1),(0,0),(0,0),(0,0)],[(2,1),(0,0),(0,0),(0,0)]]
Total: 16
With 2 fillings:
[]
Total: 0
With 3 fillings:
[]
Total: 0
With 4 fillings:
[[(-2,1),(-2,1),(-2,1),(-2,1)],[(2,3),(4,1),(3,4),(4,1)],
[(4,1),(2,3),(4,1),(2,3)],[(4,3),(4,3),(4,3),(4,3)]]
Total: 4
Candidate hyperbolic fillings:
With 1 fillings:
[]
Total: 0
With 2 fillings:
[]
Total: 0
With 3 fillings:
[]
Total: 0
With 4 fillings:
[[(2,3),(-2,1),(2,3),(4,1)],[(2,3),(4,1),(2,3),(5,1)],
[(2,3),(4,1),(3,4),(4,1)],[(3,4),(4,1),(2,3),(4,1)],
[(4,1),(-2,1),(-2,1),(-2,1)],[(4,1),(2,3),(-2,1),(2,3)],
[(4,1),(2,3),(4,1),(3,4)],[(4,1),(3,4),(4,1),(2,3)],
[(4,1),(4,3),(4,1),(2,3)],[(4,3),(4,3),(4,3),(5,4)],
[(4,3),(4,3),(5,4),(4,3)],[(5,4),(4,3),(4,3),(4,3)]]
Total: 12
\end{verbatim}
The list (I) of candidate exceptional fillings is then very short, and up to the action of
(\ref{M4_begin:eqn})-(\ref{M4_end:eqn}) it actually reduces to
$$0\quad \infty\quad (-2, -2, -2, -2)$$
which are indeed all exceptional: we have $M_4(\infty)=M_5(-1,\infty)$ and $M_4(0)=M_5(-1,\emptyset,0)$,
and we know that $\infty$ and $0$ are exceptional for $M_5$; moreover the
exceptionality of $(-2, -2, -2, -2)$ was proved using the \emph{Recognizer}~\cite{Rec}, as
discussed in Subsection~\ref{notable:subsection}. To state our classification result for the exceptional fillings of $M_4$ we now
note that under the identification
$$M_4(-1,\alpha_2,\alpha_3,\alpha_4) = M_3(\alpha_2+1,\alpha_3,\alpha_4+1)$$
the maps (\ref{magic2:eqn}) and (\ref{M3_1:eqn}) induce for $M_4$ respectively
\begin{align}
\left(-1, \alpha_2, \alpha_3, \alpha_4\right) & \mapsto \left( -1, \alpha_3-1, \alpha_2+1, \alpha_4 \right) \label{M4_1:eqn} \\
\left(-1, -2, -2, \alpha_4\right) & \mapsto \left( -1, -2, -2, -\alpha_4-4 \right). \label{M4_2:eqn}
\end{align}

\begin{prop} \label{generate:prop}
Under the correspondence
$$(\alpha_1,\alpha_2,\alpha_3)\leftrightsquigarrow(-1,\alpha_1-1,\alpha_2,\alpha_3-1)$$
the action of the maps \emph{(\ref{magic1:eqn})}-\emph{(\ref{M3_2:eqn})}
on triples is generated by the action of \emph{(\ref{M4_begin:eqn})}-\emph{(\ref{M4_2:eqn})} on $4$-tuples.
\end{prop}
\begin{proof}
It is very easy to see that (\ref{magic1:eqn}) is induced by a combination of
(\ref{M4_begin:eqn}) and (\ref{M4_second:eqn}), while (\ref{magic2:eqn})
is (\ref{M4_1:eqn}), therefore the action of (\ref{M4_begin:eqn})-(\ref{M4_2:eqn})
generates the whole isometry group (\ref{magic1:eqn})-(\ref{magic2:eqn}) of $M_3$. To show that also
(\ref{M3_begin:eqn})-(\ref{M3_2:eqn}) get generated, we first note that as a composition of
(\ref{M4_third:eqn}) and (\ref{M4_end:eqn}) we get
\begin{align}
\left(\alpha_1,\alpha_2,\alpha_3,\alpha_4\right) & \longmapsto
    \left(\tfrac{\alpha_1}{\alpha_1-1}, 2-\alpha_2, \tfrac{\alpha_3}{\alpha_3-1}, 2-\alpha_4\right).\label{M4_more:eqn}
\end{align}
We then start with (\ref{M3_begin:eqn}). Up to (\ref{magic1:eqn})-(\ref{magic2:eqn}), that
we already know to come from (\ref{M4_begin:eqn})-(\ref{M4_2:eqn}), we can rewrite it as
$$\left(\alpha_2, \tfrac 12, \alpha_3\right)
\mapsto \left(4-\alpha_3, \tfrac 12, 4-\alpha_2\right)$$
and then we can generate it as follows:
\begin{align*}
(\ref{M3_begin:eqn}):\qquad \left(\alpha_2, \tfrac 12, \alpha_3\right) \leftrightsquigarrow &
\left(-1,\alpha_2-1,\tfrac 12, \alpha_3-1\right) \mathop{\mapsto}\limits_{(\ref{M4_more:eqn})}
\left(\tfrac 12, 3-\alpha_2, -1,3-\alpha_3\right) \\
\mathop{\mapsto}\limits_{(\ref{M4_begin:eqn})^2} &
\left(-1,3-\alpha_3, \tfrac 12, 3-\alpha_2\right) \leftrightsquigarrow
\left(4-\alpha_3, \tfrac 12, 4-\alpha_2\right).
\end{align*}
We next similarly realize (\ref{M3_a:eqn}) and (\ref{M3_b:eqn}),
again only up to (\ref{magic1:eqn})-(\ref{magic2:eqn}):
\begin{align*}
(\ref{M3_a:eqn}):\qquad \left(\alpha_2, \tfrac 32, \alpha_3\right) \leftrightsquigarrow &
\left(-1,\alpha_2-1,\tfrac 32, \alpha_3-1\right)  \mathop{\mapsto}\limits_{(\ref{M4_third:eqn})}
\left(\tfrac 32, \tfrac{\alpha_2-3}{\alpha_2-2}, -1,\tfrac{\alpha_3-3}{\alpha_3-2}\right) \\
\mathop{\mapsto}\limits_{(\ref{M4_begin:eqn})^2} &
\left(-1,\tfrac{\alpha_3-3}{\alpha_3-2}, \tfrac 32, \tfrac{\alpha_2-3}{\alpha_2-2}\right) \leftrightsquigarrow
\left(\tfrac{2\alpha_3-5}{\alpha_3-2}, \tfrac 32, \tfrac{2\alpha_2-5}{\alpha_2-2}\right);
\end{align*}
\begin{align*}
\left(\ref{M3_b:eqn}):\qquad (\tfrac 52, \alpha_2, \alpha_3\right) \leftrightsquigarrow &
\left(-1, \tfrac 32, \alpha_2,\alpha_3-1\right)  \mathop{\mapsto}\limits_{(\ref{M4_third:eqn})}
\left(\tfrac 32, -1, \tfrac{\alpha_2-2}{\alpha_2-1}, \tfrac{\alpha_3-3}{\alpha_3-2}\right) \\
\mathop{\mapsto}\limits_{(\ref{M4_begin:eqn})^3} &
\left(-1,\tfrac{\alpha_2-2}{\alpha_2-1}, \tfrac{\alpha_3-3}{\alpha_3-2}, \tfrac 32\right) \leftrightsquigarrow
\left(\tfrac{2\alpha_2-3}{\alpha_2-1}, \tfrac{\alpha_3-3}{\alpha_3-2}, \tfrac 52\right).
\end{align*}
We now observe that (\ref{M4_begin:eqn})-(\ref{M4_2:eqn}) induce on triples the map
\begin{equation}\label{varphi:map}
\begin{array}{r}
\left(4, \alpha_2, \alpha_3\right) \leftrightsquigarrow
\left(-1, 3, \alpha_2,\alpha_3-1\right)  \mathop{\mapsto}\limits_{(\ref{M4_more:eqn})}
\left(\tfrac 12, -1, \tfrac{\alpha_2}{\alpha_2-1}, 3-\alpha_3\right)\\
\mathop{\mapsto}\limits_{(\ref{M4_begin:eqn})^3}
\left(-1,\tfrac{\alpha_2}{\alpha_2-1}, 3-\alpha_3, \tfrac 12\right) \leftrightsquigarrow
\left(\tfrac{2\alpha_2-1}{\alpha_2-1}, 3- \alpha_3, \tfrac 32\right)
\end{array}
\end{equation}
and it is easy to see that (\ref{M3_c:eqn}) is obtained by conjugating
(\ref{M3_a:eqn}) under (\ref{varphi:map}). Moreover we can generate (\ref{M3_2:eqn}) as
$$\left(-1, 4,\alpha_3\right) \mathop{\mapsto}\limits_{(\ref{varphi:map})}
\left(\tfrac 32, \tfrac 32, 3- \alpha_3\right) \mathop{\mapsto}\limits_{(\ref{M3_a:eqn})}
\left(\tfrac 32, 4, \tfrac{1-2\alpha_3}{1-\alpha_3}\right) \mathop{\mapsto}\limits_{(\ref{M3_c:eqn})}
\left(-1, 4, \tfrac 1{\alpha_3}\right)$$
and the conclusion eventually follows because (\ref{M3_1:eqn}) is (\ref{M4_2:eqn}).
\end{proof}

\begin{teo} \label{M4:teo}
Every exceptional filling on $M_4$ is equivalent up to a composition of the maps \emph{(\ref{M4_begin:eqn})-(\ref{M4_2:eqn})}
to a filling containing one of
$$\begin{array}{cc}
0\quad \infty\quad (-1,-2,-1)\quad (-2, -2, -2, -2)\\
(-1, -3, -2, -3)\quad (-1, -2, -3, -4).
\end{array}$$
\end{teo}

\begin{proof}
The code {\tt search\_geometric\_solutions.py} shows that the $12$ candidate hyperbolic manifolds in list (II)
are indeed hyperbolic. We know from above that
up to (\ref{M4_begin:eqn})-(\ref{M4_end:eqn})
an exceptional filling $(\alpha_1, \alpha_2, \alpha_3, \alpha_4)$
on $M_4$ either contains one of $0,\ \infty,\ (-2,-2,-2,-2)$ or
$\alpha_1=-1$ and $(\alpha_2+1,\alpha_3,\alpha_4+1)$ is exceptional for $M_3$.
In the latter case by Theorem~\ref{magic:teo} we have that
$(\alpha_2+1,\alpha_3,\alpha_4+1)$ contains one of
$$0\quad 1\quad 2\quad 3\quad \infty\quad (-1, -1)\quad (-1, -3, -3)\quad (-2, -2, -2) \quad \left(\tfrac 23, 4, 4\right)$$
up to the maps (\ref{magic1:eqn})-(\ref{M3_2:eqn}). Proposition~\ref{generate:prop} shows that the action of these
maps is generated by (\ref{M4_begin:eqn})-(\ref{M4_2:eqn}). We readily deduce that up to (\ref{M4_begin:eqn})-(\ref{M4_2:eqn})
an exceptional
$(-1,\alpha_2,\alpha_3,\alpha_4)$ for $M_4$ contains one of
$$\begin{array}{c}
(-1, -1)\quad (-1, 0)\quad (-1, 1)\quad (-1, 2)\quad (-1, \infty) \quad (-1, -2, -1) \\
(-1, -2, -3, -4)\quad (-1, -3, -2, -3) \quad  \left(-1, 3, \tfrac 23, 3\right).
\end{array}$$
We can then dismiss $(-1,0)$, $(-1,1)$, $(-1, 2)$ and $(-1,\infty)$ because we see that $0$, $1$, $2$, $\infty$ are
exceptional on $M_4$: we know that $\infty$ and $0$ are, while
$1$ is generated by $\infty$ and $2$ is generated by $0$ under both
(\ref{M4_third:eqn}) and (\ref{M4_end:eqn}).
Using (\ref{M4_1:eqn}) we can also transform
$(-1,-1)$ into $(-1,\emptyset,0)$ and then dismiss it. Finally, using (\ref{M4_more:eqn}) we can transform
$\left(-1, 3, \tfrac 23, 3\right)$ into $\left(\tfrac 12, -1, -2, -1\right)$ and dismiss it because it contains $(-1, -2, -1)$, and the
proof is complete.
\end{proof}

\begin{cor}
Every filling on $M_4$ is hyperbolic, except those listed below and those obtained from them
via compositions of the maps \emph{(\ref{M4_begin:eqn})-(\ref{M4_2:eqn})}:
\begin{align*}
M_4\left(\infty, \tfrac ab, \tfrac cd, \tfrac ef\right) & = \seiftre {S^2} abd{-c}ef \\
M_4\left(0, \tfrac ab, \tfrac cd, \tfrac ef\right) & = \seifdue {D}b{b-a}f{f-e} \bigu 0110 \seifdue D 2{-1}{c-2d}{c-d}\\
M_4\left(-1, -2, -1, \tfrac ab \right) & =  \seifuno{A}{b}{-a} \bigb 0110\\
M_4\left(-1, -2, -3, -4\right) & = \seifdue{D}212{-1} \bigu {-1}21{-1} \seifdue{D}2131 \\
M_4\left(-1, -3, -2, -3\right) & = \seifdue{D}212{-1} \bigu {-1}31{-2} \seifdue{D}2131 \\
M_4\left(-2, -2, -2, -2\right) & = \seifdue{D}212{-1} \bigu {-1}41{-3} \seifdue{D}2131.
\end{align*}
\end{cor}
\begin{proof}
We only need to check the correct expression for the filled manifold on the first three lines
depending on the parameters $a,b, \ldots, f$. The equation on the first line
follows by expressing $M_4(\infty)$ as an open chain link with 3 components. For the equation
on the second line we have
\begin{align*}
M_4\left(0, \tfrac ab, \tfrac cd, \tfrac ef\right) & = M_5\left(0, \tfrac {a-b}b, -1, \tfrac{c-d}d, \tfrac ef \right) =
M_5\left( \tfrac b{a-b}, \infty, 2, \tfrac{c-d}{c-2d}, \tfrac{f-e}f \right) \\
& = M_5\left(\infty,  \tfrac b{a-b}, \tfrac{f-e}f, \tfrac{c-d}{c-2d}, 2 \right) \\
& = \seifdue {D}b{b-a}f{f-e} \bigu 0110 \seifdue D {2}{-1}{c-2d}{c-d}
\end{align*}
using (\ref{transposition:eqn}) and Corollary~\ref{main:cor}. Finally the equation on the third line
also follows from Corollary~\ref{main:cor} or~\cite{MaPe}.
\end{proof}

\subsection{The minimally twisted 5-chain link} \label{M5:subsection}
We eventually prove here Theorem~\ref{main:teo}, concerning
the isolated exceptional fillings on $M_5$. Recall first that $M_5$ decomposes into 10 regular
ideal hyperbolic tetrahedra, and this decomposition is totally symmetric, so each
cusp section decomposes into $8$ equilateral triangles. More precisely, with respect to the meridian-longitude
homology basis of the cusp, its shape is given by $-\tfrac12 + i\tfrac{\sqrt 3}2$, and it is quite easy
to see that the area of each cusp in a maximal horospherical cusp
section is equal to $A = 2\sqrt 3$, because each individual equilateral triangle has Euclidean area $\tfrac{\sqrt3}4$.
The length of a slope $\tfrac pq$ is hence
$$\ell\left(\tfrac pq\right) = \sqrt{4\big( \left(p-\tfrac q2\right)^2 +3\left(\tfrac q2\right)^2 \big)} = 2\sqrt{p^2 +q^2 - pq}.$$
The slopes having length at most 6 are therefore
$$\infty\quad -2\quad -1\quad -\tfrac12\quad 0\quad \tfrac13\quad \tfrac12\quad\tfrac23\quad 1\quad \tfrac32\quad 2\quad 3.$$
Recall now that the action on slopes of the symmetry group of $M_5$ is generated by
(\ref{first:eqn})-(\ref{transposition:eqn}); of course
(\ref{first:eqn}) and (\ref{firstbis:eqn}) act trivially on the set of slopes of length less than $6$
just enumerated, but (\ref{transposition:eqn}) allows to group them as
$$\{\infty,\ 0,\ 1\},\qquad \left\{-1,\ \tfrac 12,\ 2\right\},\qquad \left\{-2,\ -\tfrac 12,\ \tfrac 13,\ \tfrac 23,\ \tfrac 32, 3\right\}.$$
Now, we already know that the slopes in the first set are exceptional and those in the second set
are not, since they give $M_4$ as a filling. SnapPy tells us that $M_5(-2)$ is hyperbolic
so the slopes in the third set are again non-exceptional. Since we know the isolated exceptional fillings on $M_4$,
to understand those on $M_5$ we are only left to understand those on $M_5(-2)$, which
we can do feeding $M_5(-2)$ to our code. But, to avoid considering again the slopes
$\infty,\ 0,\ 1$ that we know to be exceptional on $M_5$, and those coming from
exceptional slopes on $M_4$, we put the preamble
\begin{verbatim}
exclude = [
[(1,0),(0,0),(0,0),(0,0)],[(0,1),(0,0),(0,0),(0,0)],
[(1,1),(0,0),(0,0),(0,0)],[(0,0),(1,0),(0,0),(0,0)],
[(0,0),(0,1),(0,0),(0,0)],[(0,0),(1,1),(0,0),(0,0)],
[(0,0),(0,0),(1,0),(0,0)],[(0,0),(0,0),(0,1),(0,0)],
[(0,0),(0,0),(1,1),(0,0)],[(0,0),(0,0),(0,0),(1,0)],
[(0,0),(0,0),(0,0),(0,1)],[(0,0),(0,0),(0,0),(1,1)],
[(-1,1),(0,0),(0,0),(0,0)],[(2,1),(0,0),(0,0),(0,0)],
[(1,2),(0,0),(0,0),(0,0)],[(0,0),(-1,1),(0,0),(0,0)],
[(0,0),(2,1),(0,0),(0,0)],[(0,0),(1,2),(0,0),(0,0)],
[(0,0),(0,0),(-1,1),(0,0)],[(0,0),(0,0),(2,1),(0,0)],
[(0,0),(0,0),(1,2),(0,0)],[(0,0),(0,0),(0,0),(-1,1)],
[(0,0),(0,0),(0,0),(2,1)],[(0,0),(0,0),(0,0),(1,2)]]
\end{verbatim}
The output of our code is the following:
\begin{verbatim}
Candidate exceptional fillings:
With 1 fillings:
[]
Total: 0
With 2 fillings:
[]
Total: 0
With 3 fillings:
[]
Total: 0
With 4 fillings:
[[(-2,1),(-2,1),(-2,1),(-2,1)],[(-2,1),(1,3),(3,1),(1,3)],
[(-1,2),(-2,1),(3,2),(3,2)],[(-1,2),(3,1),(3,1),(-1,2)],
[(1,3),(3,1),(1,3),(-2,1)],[(1,3),(3,2),(3,2),(1,3)],
[(3,2),(3,2),(-2,1),(-1,2)]]
Total: 7
Candidate hyperbolic fillings:
With 1 fillings:
[]
Total: 0
With 2 fillings:
[]
Total: 0
With 3 fillings:
[]
Total: 0
With 4 fillings:
[[(-2,1),(-2,1),(3,1),(-2,1)],[(-2,1),(1,3),(3,1),(2,3)],
[(-2,1),(3,1),(3,1),(-2,1)],[(-1,2),(-2,1),(3,2),(3,1)],
[(-1,2),(-1,2),(3,2),(3,2)],[(-1,2),(3,1),(3,1),(-2,1)],
[(-1,3),(1,3),(3,2),(2,3)],[(-1,3),(3,1),(-2,1),(-2,1)],
[(-1,3),(3,1),(-1,2),(-2,1)],[(1,3),(3,2),(2,3),(2,3)],
[(1,3),(4,3),(2,3),(2,3)],[(2,3),(2,3),(3,1),(1,3)],
[(3,1),(1,3),(3,1),(1,3)],[(3,2),(3,1),(-1,2),(-2,1)],
[(3,2),(3,1),(-1,2),(-1,2)]]
Total: 15
\end{verbatim}

To achieve tasks (i) and (ii) we then have to show that the $7$ candidate exceptional
(closed) fillings in list (I) are indeed exceptional, and that the $15$ candidate hyperbolic
(closed) fillings in list (II) are indeed hyperbolic. For task (i) we take into account the
maps (\ref{first:eqn})-(\ref{transposition:eqn}), under which the $7$ fillings reduce to
$$(-2, -2, -2, -2, -2)\qquad \left(-2, -\tfrac 12, 3, 3, - \tfrac 12\right).$$
The \emph{Recognizer}~\cite{Rec} then confirms that both these fillings are
exceptional and give rise to the graph manifolds
described in Corollary~\ref{main:cor}.
Task (ii) is achieved directly by running the code
{\tt search\_geometric\_solutions.py} on the $15$ candidate hyperbolic manifolds.

To conclude the proof of Theorem~\ref{main:teo} we are only left to enumerate up to
the action of (\ref{first:eqn})-(\ref{figure8:eqn}) the exceptional fillings
of $M_5$ coming from the exceptional ones on $M_4=M_5(-1)$ determined in Theorem~\ref{M4:teo}.
Recall first that
$$M_5(-1,\alpha_2, \alpha_3, \alpha_4, \alpha_5) = M_4(\alpha_2+1, \alpha_3, \alpha_4, \alpha_5+1)$$
and that up to this identification the maps (\ref{first:eqn})-(\ref{figure8:eqn}) generate
the maps (\ref{M4_begin:eqn})-(\ref{M4_2:eqn}) induced by the symmetries of $M_4$.
So up to (\ref{M4_begin:eqn})-(\ref{M4_2:eqn}) we see that $(-1,\alpha_2, \alpha_3, \alpha_4, \alpha_5)$
is exceptional for $M_5$ if and only if $(\alpha_1+1, \alpha_2, \alpha_3, \alpha_4+1)$ contains one of
$$\begin{array}{ccc}
0 & \infty & (-1, -2, -1)\\ (-2, -2, -2, -2) & (-1, -3, -2, -3) & (-1, -2, -3, -4)
\end{array}$$
\emph{i.e.}, if and only if $(-1,\alpha_2, \alpha_3, \alpha_4, \alpha_5)$ contains one of
$$\begin{array}{ccc}
(-1,-1) & (-1, \infty) & (-1, -2, -2, -1)\\
(-1, -3, -2, -2, -3) &  (-1, -2, -3, -2, -4) & (-1, -2, -2, -3, -5)
\end{array}$$
but we can dismiss $(-1,-1)$ and $(-1,\infty)$ because the latter contains $\infty$,
and the former does up to (\ref{first:eqn})-(\ref{figure8:eqn}).
The proof is now complete.

\subsection{Computer time}
As thoroughly explained above,
during our investigation we have used two different programs: the main python code {\tt find\_exceptional\_fillings} and a shorter python code named {\tt search\_geometric\_solutions}. Both codes are
available from~\cite{M}, and the computer time spent to run each of them for each manifold $M_2, \ldots, M_5$ is shown in
Table~\ref{computer_details:table}. As one can see, once properly organized as we have described, the search requires very little computer time.

\begin{table}
\begin{center}
\begin{tabular}{r||c|c|c|c}
 & $M_2$ & $M_3$ & $M_4$ & $M_5$ \\
\hline \hline
{\tt find\_exceptional\_fillings} & $1''$ & $24''$ & $2'\,10''$ & $3'\,48''$  \\
{\tt search\_geometric\_solutions} & $<1''$ & $<1''$ & $<1''$ & $<1''$
\end{tabular}
\end{center}
\caption{Computer time needed by each code to classify the exceptional fillings of $M_2, \ldots, M_5$.}
\label{computer_details:table}
\end{table}

\section{Tables}
In this section we expand Theorems~\ref{main:teo} and~\ref{M4:teo}, listing all the isolated exceptional fillings
on $M_4$ and $M_5$, up to the action of their isometry groups. The results stated here are summarized by the
entries in Table~\ref{exceptional_2:table}, while those in Table~\ref{exceptional_1:table} were
obtained \emph{a posteriori} using the action of the isometry groups.
We also show all the filled manifolds, to do which,
in addition to the notation for Seifert and graph manifolds introduced in Section \ref{graph:subsection}, we will use
$T_X$ to denote the torus-bundle on $S^1$ obtained from $T\times[0,1]$ by gluing
$T\times\{0\}$ to $T\times\{1\}$ along $X\in\GL(2,\matZ)$ with respect to
parallel homology bases; note that for $T_X$ to be orientable now one needs to have
$\det(X)=+1$.

\begin{teo}\label{M4_large:teo}
The isolated exceptional fillings on $M_4$, seen up to the action of the isometry group of $M_4$
generated by \emph{(\ref{M4_begin:eqn})-(\ref{M4_end:eqn})}, are those listed
in Tables~\ref{M4:table} and~\ref{M4_2:table}.
These fillings are pairwise inequivalent under \emph{(\ref{M4_begin:eqn})-(\ref{M4_end:eqn})}.
\end{teo}

\begin{table}
\begin{center}
\begin{tabular}{c|c|c}
$k$ & Exceptional fillings & Filled manifold \\
\hline\hline
 1 & $\left(\infty\right)$ & \alza $P\times S^1$ \\
\cline{2-3}
   & $\left(0\right)$ & \alza $\big(P\times S^1\big) \bigu 0110 \seifuno A21$ \\
\hline
 2 & \alza $\left(-1,-1\right)$  &  $\seifdue D 2{-1}31 \bigu 0110 \left(P\times S^1\right)$ \\
\cline{2-3}
 & \alza $\left(-1,\emptyset,3\right)$ & $\seifuno A21 \bigu 0110 \seifuno A21$ \\
\hline
 3 & \begin{tabular}{cc} \alza $\left(-1,-2, -1\right)$ & \alza $\left(-2,-1,-2\right)$ \end{tabular} & $\big(P\times S^1\big) \bigb 0110$ \\
\cline{2-3}
 & \begin{tabular}{cc} \alza $\left(-1, -\tfrac 12, 4\right)$ & \alza $\left(-1,\tfrac 12, \tfrac 52\right)$ \end{tabular} &
 $\seifdue D 2{-1}31 \bigu 0110 \seifuno A2{-1}$ \\
\end{tabular}
\end{center}
\caption{Non-closed isolated exceptional fillings on $M_4$, split according to the
number $k$ of filled slopes, up to the action of the isometry group of $M_4$.}
\label{M4:table}
\end{table}

\begin{table}
\begin{center}
\begin{tabular}{c|c}
Exceptional fillings & Filled manifold \\
\hline\hline
\begin{tabular}{c} \alza $\left(-1, -2, -2, -5\right)$ \\ $\left(-1, -3, -1, -5\right)$ \\ \alza $\left(-1, -2, -4, -3\right)$ \end{tabular}
& $\seiftre {S^2}31 31 4{-3}$ \\
\hline
\begin{tabular}{c}
\alza $\left(-1, -2, -2, -4\right)$ \\ $\left(-1, -3, -1, -4\right)$ \\ \alza $\left(-1, -2, -3, -3\right)$ \end{tabular}
& $\seiftre {S^2}2{-1} 41 51$ \\
\hline
\begin{tabular}{c}
\alza $\left(-1, -2, -2, -3\right)$ \\ $\left(-1, -3, -1, -3\right)$ \end{tabular}
& $\seiftre {S^2}2{-1} 31 71$ \\
\hline
\begin{tabular}{c} \alza $\left(-1, -2, -2, -6\right)$ \\ $\left(-1, -3, -1, -6\right)$ \\ \alza $\left(-1, -3, -5, -2\right)$ \end{tabular}
& $\seifdue D212{-1} \bigu {-1}110 \seifdue D2131$ \\
\hline
\begin{tabular}{c} \alza $\left(-1, \tfrac 12, \tfrac 83, \tfrac 12\right)$ \\ $\left(-1, -2, 4, -\tfrac 23\right)$ \\ \end{tabular}
&    $\seifdue{D}212{-1}
        \bigcup\nolimits_{{\tiny{\matr 110{-1}}}\phantom{\Big|}\!\!}
        \seifdue{D}2131\phantom{\Big|}$ \\
\hline
\alza $\left(-1, \tfrac 23, \tfrac 52, \tfrac 23\right)$
& $\seifdue{D}212{-1}
        \bigcup\nolimits_{{\tiny{\matr 21{-1}{-1}}}\phantom{\Big|}\!\!}
        \seifdue{D}2131\phantom{\Big|}$ \\
\hline
\begin{tabular}{c}
\alza $\left(-1, -2, -3, -4\right)$ \\ $\left(-1, -4, -1, -4\right)$ \end{tabular}
&     $\seifdue{D}212{-1}
        \bigcup\nolimits_{{\tiny{\matr {-1}21{-1}}}\phantom{\Big|}\!\!}
        \seifdue{D}2131\phantom{\Big|}$ \\
\hline
$\left(-1, -3, -2, -3\right)$ & $\seifdue{D}212{-1}
      \bigcup\nolimits_{{\tiny{\matr {-1}31{-2}}}\phantom{\Big|}\!\!}
        \seifdue{D}2131\phantom{\Big|}$ \\
\hline
\alza $\left(-2, -2, -2, -2\right)$
& $\seifdue{D}212{-1}
        \bigcup\nolimits_{{\tiny{\matr {-1}41{-3}}}\phantom{\Big|}\!\!}
        \seifdue{D}2131\phantom{\Big|}$ \\
\hline
\alza $\left(-1, \tfrac 12, \tfrac 32, 3\right)$
& $T_{\tiny{\matr {-3}1{-1}0}}$ \\
\hline
\alza $\left(-1, 4, 5, -\tfrac 12\right)$
& $\seifuno A 21 \bigb 0110$ \\
\hline
\begin{tabular}{c}
\alza $\left(-1, 3, 4, -\tfrac 13\right)$ \end{tabular}
& $\seifuno{A}21\big/_{{\tiny{\matr 1110}}\phantom{\Big|}}\phantom{\Big|}$ \\
\hline
\alza $\left(-1,  \tfrac 32, \tfrac 52, \tfrac 13\right)$ & $\seifuno A 21 \bigb 2110$
\end{tabular}
\end{center}
\caption{Closed isolated exceptional fillings on $M_4$ up to the action of the isometry group of $M_4$.}
\label{M4_2:table}
\end{table}

\begin{proof}
The discussion in Subsection~\ref{M4:subsection} shows that, up to the action
(\ref{M4_begin:eqn})-(\ref{M4_end:eqn}) of the isometry group of $M_4$,
an isolated exceptional filling on $M_4$ either contains $-1$ or is equal to $0$, $\infty$,
or $\left(-2,-2,-2,-2\right)$. If it contains $-1$ then it is of type $\left(-1, \alpha_2 -1, \alpha_3, \alpha_4-1\right)$
where $\left(\alpha_2, \alpha_3, \alpha_4\right)$ is an isolated exceptional filling on the magic manifold $M_3$,
as described in Theorem~\ref{magic:teo}. But now we are \emph{not} identifying slopes on $M_4$
equivalent under the maps (\ref{M4_1:eqn})-(\ref{M4_2:eqn}) induced
by isometries of $M_3=M_4(-1)$ or $M_2=M_4(-1,-2)$ or $M_1=M_4(-1,-2,-2)$, therefore
we must take the slopes listed in Theorem~\ref{magic:teo}, consider their
full orbit under the isometries (\ref{magic1:eqn})-(\ref{magic2:eqn}), pull them back to
$M_4$, and then remove those that are not isolated on $M_4$ and
mod out under (\ref{M4_begin:eqn})-(\ref{M4_end:eqn}).
The process is long but straight-forward and leads to the tables, with the
manifolds always identified by hand and/or using the \emph{Recognizer}.
\end{proof}

\begin{teo} \label{M5_large:teo}
The isolated exceptional fillings on $M_5$, seen up to the action of the isometry group of $M_5$ generated by
\emph{(\ref{first:eqn})-(\ref{transposition:eqn})}, are
those listed in Tables~\ref{M_5:table},~\ref{M_5_2:table}, and~\ref{M_5_3:table}.
These fillings are pairwise inequivalent under \emph{(\ref{first:eqn})-(\ref{transposition:eqn})}.
\end{teo}
\begin{proof}
The scheme of the proof is similar, and we omit all details (carried out using a dedicated code).
The isolated exceptional fillings on $M_5$ described in
Theorem~\ref{main:teo} not coming from $M_4$, namely not containing
a $-1$ slope up to (\ref{first:eqn})-(\ref{transposition:eqn}), contribute directly to the tables.
Those coming from $M_4$ are acted on using the full isometry group of $M_4$,
pulled back to $M_5$, depurated from the non-isolated ones, and modded out
under the isometry group of $M_5$. Again the manifolds are identified
by hand and/or using the \emph{Recognizer}.
\end{proof}

\begin{table}
\begin{center}
\begin{tabular}{c|c|c}
$k$ & Exceptional fillings & Filled manifold \\
\hline\hline
 1 & $\left(1\right)$ & \alza $\big(P\times S^1\big) \bigu 0110 \big(P\times S^1\big)$ \\
\hline
 2 & $\left(-1,-1\right)$ & \alza $\big(P\times S^1\big) \bigu 0110 \seifuno A21$ \\
\hline
 3 & \begin{tabular}{c} \alza $\left(-1,-2,-1\right)$ \\ \alza $\left(-2,-1,-2\right)$ \end{tabular} &
    $\big(P\times S^1\big) \bigu 0110 \seifdue D2131$ \\
\cline{2-3}
 & \alza $\left(\tfrac 12,3, \tfrac 13\right)$ & $\seifuno A21 \bigu 0110 \seifuno A21$ \\
\hline
 4 & \begin{tabular}{c} \alza $\left(-1,-2, -2, -1\right)$ \\ \alza $\left(-1,-3,-1,-2\right)$
        \\ \alza $\left(-2,-2,-1,-3\right)$ \end{tabular} & $\big(P\times S^1\big) \bigb 0110$ \\
\cline{2-3}
 & \begin{tabular}{c} \alza $\left(-1,-2, \tfrac 12, \tfrac 52\right)$ \\ \alza $\left(-2, -1,-\tfrac 12, \tfrac 52\right)$ \\
 \alza $\left(-1,-2, -\tfrac 12, 4\right)$ \\ \alza $\left(-2, -1,-\tfrac 32, 4\right)$ \end{tabular} &
 $\seifdue D 2{-1}31 \bigu 0110 \seifuno A2{-1}$
\end{tabular}
\end{center}
\caption{Non-closed isolated exceptional fillings on $M_5$,
split according to the number $k$ of filled slopes,
up to the action of the isometry group of $M_5$.}
\label{M_5:table}
\end{table}

\begin{table}
\begin{center}
\begin{tabular}{c|c}
 Exceptional fillings & Filled manifold \\
\hline\hline
\begin{tabular}{c}
\alza $\left(-1, -3, -1, -5, -3\right)$ \\ $\left(-1, -4, -1, -5, -2\right)$ \\
\alza $\left(-1, -2, -2, -2, -6\right)$ \\ $\left(-1, -2, -3, -1, -6\right)$ \\
\alza $\left(-1, -2, -2, -4, -4\right)$ \\ $\left(-1, -2, -5, -2, -3\right)$ \\
\alza $\left(-1, -2, -3, -4, -3\right)$ \\ \end{tabular}
& $\seiftre {S^2}31 31 4{-3}$ \\
\hline
\begin{tabular}{c}
\alza $\left(-1, -2, -2, -2, -5\right)$ \\ $\left(-1, -2, -3, -1, -5\right)$ \\
\alza $\left(-1, -2, -2, -3, -4\right)$ \\ $\left(-1, -2, -4, -1, -4\right)$ \\
\alza $\left(-1, -3, -1, -4, -3\right)$ \\ $\left(-1, -2, -3, -3, -3\right)$ \\
\alza $\left(-1, -2, -4, -2, -3\right)$ \end{tabular}
& $\seiftre {S^2}2{-1} 41 51$ \\
\hline
\begin{tabular}{c}
\alza $\left(-1, -2, -2, -2, -4\right)$ \\ $\left(-1, -2, -3, -1, -4\right)$ \\
\alza $\left(-1, -3, -1, -3, -3\right)$ \\ $\left(-1, -2, -1, -3, -2\right)$ \end{tabular}
& $\seiftre {S^2}2{-1} 31 71$ \\
\hline
 \begin{tabular}{c}
 \alza $\left(-2, -2, -2, -1, -7\right)$ \\ $\left(-1, -2, -3, -1, -7\right)$ \\
 \alza $\left(-1, -2, -3, -5, -3\right)$ \\ $\left(-1, -2, -6, -2, -3\right)$ \\
 \alza $\left(-1, -3, -1, -6, -3\right)$ \\ $\left(-1, -4, -1, -6, -2\right)$ \\
\alza $\left(-1, -2, -2, -5, -4\right)$ \end{tabular}
& $\seifdue D212{-1} \bigu {-1}110 \seifdue D2131$ \\
\hline
\begin{tabular}{c}
\alza $\left(-1, -2, -\tfrac 23, 4, -3\right)$ \\ $\left(-1, -2, -2, 4, -\tfrac 53\right)$ \\
\alza $\left(-1, -2,  3, -1, -\tfrac 53\right)$ \\  $\left(-1, -\tfrac 12, -1, \tfrac 12, \tfrac 53\right)$ \end{tabular}
&    $\seifdue{D}212{-1}
        \bigcup\nolimits_{{\tiny{\matr 110{-1}}}\phantom{\Big|}\!\!}
        \seifdue{D}2131\phantom{\Big|}$ \\
\end{tabular}
\end{center}
\caption{Closed isolated exceptional fillings on $M_5$ up to the action of the isometry group of $M_5$ (part I).}
\label{M_5_2:table}
\end{table}

\begin{table}
\begin{center}
\begin{tabular}{c|c}
Exceptional fillings & Filled manifold \\
\hline\hline
\begin{tabular}{c}
\alza $\left(-1, -\tfrac 13, -1, \tfrac 23, \tfrac 32\right)$ \\ $\left(-1, -\tfrac 13, \tfrac 52, \tfrac 23, -2\right)$ \end{tabular}
& $\seifdue{D}212{-1}
        \bigcup\nolimits_{{\tiny{\matr 21{-1}{-1}}}\phantom{\Big|}\!\!}
        \seifdue{D}2131\phantom{\Big|}$ \\
\hline
\begin{tabular}{c}
\alza $\left(-1, -3, -1, -4, -4\right)$ \\ $\left(-1, -2, -2, -3, -5\right)$ \\
\alza $\left(-1, -2, -4, -1, -5\right)$ \\ $\left(-1, -2, -4, -3, -3\right)$ \end{tabular}
&     $\seifdue{D}212{-1}
        \bigcup\nolimits_{{\tiny{\matr {-1}21{-1}}}\phantom{\Big|}\!\!}
        \seifdue{D}2131\phantom{\Big|}$ \\
\hline
\begin{tabular}{c}
\alza $\left(-1, -2, -3, -2, -4\right)$ \\ $\left(-1, -3, -3, -1, -4\right)$ \end{tabular}
& $\seifdue{D}212{-1}
      \bigcup\nolimits_{{\tiny{\matr {-1}31{-2}}}\phantom{\Big|}\!\!}
        \seifdue{D}2131\phantom{\Big|}$ \\
\hline
\alza $\left(-1, -3, -2, -2, -3\right)$
& $\seifdue{D}212{-1}
        \bigcup\nolimits_{{\tiny{\matr {-1}41{-3}}}\phantom{\Big|}\!\!}
        \seifdue{D}2131\phantom{\Big|}$ \\
\hline
\alza $\left(-2, -2, -2, -2, -2\right)$
& $\seifdue{D}212{-1}
        \bigcup\nolimits_{{\tiny{\matr {-1}51{-4}}}\phantom{\Big|}\!\!}
        \seifdue{D}2131\phantom{\Big|}$ \\ \hline
\begin{tabular}{c}
\alza $\left(-1, -2, -1, \tfrac 12, \tfrac 12\right)$ \\ $\left(-1, \tfrac 12, 3, -1, -\tfrac 12\right)$ \\
\alza $\left(-2, -1, -\tfrac 12, \tfrac 32, 3\right)$ \\ \end{tabular}
& $T_{\tiny{\matr {-3}1{-1}0}}$ \\
\hline
\begin{tabular}{c}
\alza $\left(-1, -2, -\tfrac 12, 5, 3\right)$ \\ $\left(-1, -2, 4, 5, -\tfrac 32\right)$ \\
\alza $\left(-1, 4, 4, -1, -\tfrac 32\right)$ \\ \end{tabular}
& $\seifuno A 21 \bigb 0110$ \\
\hline
\begin{tabular}{c}
\alza $\left(-2, -1, 2, 4, -\tfrac 13\right)$ \\ $\left(-1, -2, 3, 4, -\tfrac 43\right)$ \\
\alza $\left(-1, -2, \tfrac 12, \tfrac 73, \tfrac 13\right)$ \end{tabular}
& $\seifuno{A}21\big/_{{\tiny{\matr 1110}}\phantom{\Big|}}\phantom{\Big|}$ \\
\hline
\begin{tabular}{c}
\alza $\left(\tfrac 12, -1, -2, \tfrac 13, \tfrac 52\right)$ \\ $\left(-1, -2, \tfrac 32, \tfrac 52, -\tfrac 23\right)$ \\
\alza $\left(-1, \tfrac 12, -1, \tfrac 13, \tfrac 32\right)$ \end{tabular}
& $\seifuno A 21 \bigb 2110$ \\
\hline
\alza $\left(-2, -\tfrac 12, 3, 3, -\tfrac 12\right)$
& $\seifuno A 2{-1} \bigb 1211$ \\
\end{tabular}
\end{center}
\caption{Closed isolated exceptional fillings on $M_5$  up to the action of the isometry group of $M_5$ (part II).}
\label{M_5_3:table}
\end{table}

\vspace{1cm}

{\small
\noindent
\hspace{6cm}
\vbox{\noindent Dipartimento di Matematica\\
Largo Pontecorvo 5\\
56127 Pisa, Italy\\
martelli at dm dot unipi dot it\\
petronio at dm dot unipi dot it}}

\vspace {.5 cm}

{\small
\noindent
\hspace{6cm}
\vbox{\noindent School of Mathematics and Statistics \\
University of Sheffield \\
Hicks Building \\
Hounsfield Road \\
Sheffield, S3 7RH \\
United Kingdom \\
roukema at gmail dot com}}

\end{document}